\documentclass[11pt,a4paper]{article}

\usepackage[T1]{fontenc}
\usepackage[utf8]{inputenc}
\usepackage{lmodern}
\usepackage{amsmath,amssymb,amsthm,mathtools}
\usepackage{geometry}
\usepackage{microtype}
\usepackage{enumitem}
\setlist{topsep=3pt,itemsep=1pt,parsep=0pt}
\usepackage{hyperref}
\usepackage{mathrsfs}
\usepackage{bm}
\hypersetup{colorlinks=true,linkcolor=blue,citecolor=blue,urlcolor=blue}

\newtheorem{theorem}{Theorem}[section]
\newtheorem{lemma}[theorem]{Lemma}
\newtheorem{proposition}[theorem]{Proposition}
\newtheorem{corollary}[theorem]{Corollary}
\newtheorem{remark}[theorem]{Remark}
\newtheorem{definition}[theorem]{Definition}

\newtheorem*{standingassumption}{Standing assumption}

\newcommand{\R}{\mathbb R}
\newcommand{\M}{\mathcal M}

\newcommand{\diam}{\operatorname{diam}}

\newcommand{\eps}{\varepsilon}

\newcommand{\address}[1]{\def\theaddress{#1}}
\newcommand{\email}[1]{\def\theemail{#1}}
\newcommand{\printauthorinfo}{%
  \begin{center}
  \small
  \theaddress\\
  \href{mailto:\theemail}{\texttt{\theemail}}
  \end{center}
}

\title{Quantitative Gibbs Transfer in Fiber-Contracting Extensions:\ Thermodynamic Formalism, Livšic Theory, and Sharp Logarithmic Distortion}
\author{G. Ferreira}
\address{ Department of Mathematics, Federal University of Maranhão, Av. dos Portugueses 1966, Vila Bacanga, 65065--545 S\~ao Lu\'is, MA, Brazil}
\email{giovane.ferreira@ufma.br}
\date{}

\begin{document}
\maketitle
\printauthorinfo

\begin{center}

\end{center}

\begin{abstract}
We study the transfer of Gibbs estimates between a compact dynamical
system and a factor with uniformly contracting fibers. A cohomological
identity relates the orbit sums; the remaining issue is the comparison
of Bowen-ball masses. Subexponential comparison errors give directional
weak Gibbs transfer, while bounded errors preserve strong estimates.
We also establish a blockwise sequential version along prescribed
Gibbs-time sequences, with errors measured against the full block length.
Geometric and conditional-mass hypotheses provide practical criteria
for these comparisons.

A stationary renewal model shows that the distinction between bounded
and subexponential comparison is essential. Its invariant lift to a
contracting Cantor-fiber extension is not supported on a measurable
graph, and the optimal uniform forward logarithmic error is
$\log n+O(1)$. This remains true under fixed changes of the base
radius. The reverse comparison, along a specified section, has bounded
error. We also give an example in which pointwise subexponential bounds
cannot be made uniform.

For context, projection gives an affine bijection of invariant
probability measures and preserves entropy. When the potentials are
cohomologous, invariant averages and pressure agree, as do minimizing
measures and equilibrium states in the finite-pressure regime.
For H\"older potentials, a continuous section yields a stable-direction
reduction, with H\"older regularity under additional assumptions.
Subactions and Liv\v{s}ic solutions can then be lifted whenever the
corresponding quotient hypotheses hold.
\end{abstract}

\section{Introduction}

Cohomological reduction is a familiar way to remove dependence on
contracting directions. It identifies orbit sums up to a bounded
coboundary and, consequently, preserves invariant averages. Gibbs
estimates require more: even when two potentials are cohomologous,
the measures of their Bowen balls need not be comparable by a uniform
constant. This paper examines that discrepancy and its effect on the
transfer of Gibbs properties.

We formulate forward and backward criteria in terms of logarithmic
Bowen-mass errors for a corresponding pair of invariant measures.
An error of $o(n)$ suffices for uniform weak Gibbs transfer in the
specified direction. For sequential estimates, the relevant error is
$o(n_i)$ along each prescribed Gibbs block of length $n_i$, including
its admissible suffixes. This gives a blockwise Bowen sequential weak
Gibbs property; it does not imply suffix-normalized $o(n_i-j)$ control.
If the comparison errors remain bounded, strong Bowen sequential
estimates transfer along the same sequence. We give geometric and
conditional-mass conditions under which these assumptions can be checked.

The sharpness of the distinction appears in a stationary renewal
example. The base measure has polynomial gap tails, and its invariant
lift to a contracting Cantor-fiber system is not concentrated on a
measurable graph. The smallest uniform forward logarithmic comparison
error grows as $\log n+O(1)$, even after fixed adjustments of the base
radius. Along a prescribed section the backward error stays bounded.
Thus the exact correspondence of invariant measures, entropy, and
pressure does not force bounded distortion of Bowen-ball masses.
A separate example shows why pointwise control cannot replace the
uniform estimates used in the transfer theorems.

The cohomological and variational ingredients have established
antecedents. Stable-direction normalization goes back to the
Sinai-type cohomology method \cite{Sinai}; entropy preservation is
obtained here using the relative variational principle of
Ledrappier--Walters \cite{LedrappierWalters}. The necessary vanishing
of fiber entropy is proved below. Subcohomology and Liv\v{s}ic
conclusions are conditional on the corresponding results for the
quotient, including the regularity assumptions in the zooming setting
of Mbarki--Santana \cite{MbarkiSantana}.

Ferreira--Ramos \cite{FerreiraRamos} treat stable normalization,
Gibbs lifting, and equilibrium correspondence for semiconjugate
attractors. Our emphasis is on directional, measure-dependent mass
comparisons and the size of the resulting Gibbs error. The blockwise
sequential statements complement the symbolic factor results of
Ferreira--Oliveira \cite{FerreiraOliveira}; a finite-window lemma
relates symbolic cylinders to Bowen balls. The renewal construction
uses a continuous non-H\"older potential, in the tradition of
Hofbauer's examples \cite{Hofbauer}. Fern\'andez, Gallo, and
Maillard \cite{FernandezGalloMaillard} use renewal structure to show
that a regular $g$-measure need not be Gibbsian for a two-sided
specification. Their setting differs from the comparison of Bowen-ball
masses across a contracting extension considered here. Further
connections, including the results of Ranu \cite{Ranu}, are discussed
in Section~\ref{sec:comparison}.

Section~\ref{sec:geom} states the hypotheses and main results.
Sections~\ref{sec:reduction}--\ref{sec:pressure} develop the
cohomological and variational tools. The transfer theorems appear in
Sections~\ref{sec:weak-gibbs} and~\ref{sec:sequential-gibbs}, followed
by the subcohomology and Liv\v{s}ic results. Section~\ref{sec:examples}
contains the examples, including the sharp renewal estimate;
Section~\ref{sec:comparison} compares the results with earlier work.

\section{Setting and Main Results}\label{sec:geom}

\subsection{Geometric hypotheses and notation}\label{ass:selector}

Let $N$ and $M$ be compact metric spaces. Let $F:N\to N$ be continuous and
$f:M\to M$ be a local homeomorphism. Assume that $\pi:N\to M$ is a
continuous surjection satisfying
\begin{equation}
    \pi\circ F=f\circ\pi.
\end{equation}
Then
\begin{equation}
    F(N_x)\subset N_{f(x)}.
    \label{eq:fiber-map}
\end{equation}

\begin{standingassumption}[Uniform fiber contraction]
There exists $0<\lambda<1$ such that
\[
d_N(F^n\hat x,F^n\hat y)
\leq\lambda^n d_N(\hat x,\hat y)
\]
for all $n\geq0$ and all $\hat x,\hat y$ in the same fiber.
\end{standingassumption}

\begin{quote}\textbf{Additional hypothesis (continuous fiber selector).}\par

There exists a continuous map
\[
r:N\to N
\]
such that
\begin{equation}
r(\hat x)\in N_{\pi(\hat x)}
\end{equation}
and
\begin{equation}
\pi(\hat x)=\pi(\hat y)
\quad\Longrightarrow\quad
r(\hat x)=r(\hat y).
\end{equation}
\end{quote}

Thus $r$ is equivalently a continuous section of $\pi$ written on $N$.
Indeed, since $\pi$ is a continuous surjection from the compact space $N$
onto the Hausdorff space $M$, it is a quotient map. Because $r$ is continuous
and constant in the fibers of $\pi$, it factors uniquely through a continuous
map $s:M\to N$ such that
\[
\pi\circ s=\operatorname{id}_M,
\qquad
r=s\circ\pi.
\]
Conversely, every continuous section $s$ defines such a selector by
$r=s\circ\pi$. An arbitrary factor map
need not admit a continuous section.
The invariant-section case $F\circ s=s\circ f$ is treated in
Proposition~\ref{prop:invariant-graph-bowen}.

For $0<\beta\leq1$, define
\[
[\varphi]_\beta
=
\sup_{\hat x\neq\hat y}
\frac{|\varphi(\hat x)-\varphi(\hat y)|}
{d_N(\hat x,\hat y)^\beta}.
\]

We shall also use invariant measure spaces
\[
\M_F(N)=\{\widetilde\mu:F_*\widetilde\mu=\widetilde\mu\},
\qquad
\M_f(M)=\{\mu:f_*\mu=\mu\}.
\]
For a continuous potential $\varphi$ define
\begin{equation}
m(\varphi,F)
=
\min_{\widetilde\mu\in\M_F(N)}
\int_N\varphi\,d\widetilde\mu,
\label{eq:m-total}
\end{equation}
and similarly
\begin{equation}
m(\psi,f)
=
\min_{\mu\in\M_f(M)}
\int_M\psi\,d\mu.
\label{eq:m-base}
\end{equation}

\subsection{General reduction and transfer principle}

For a continuous map $T:X\to X$, write
\[
 B_T(x,\eps,n)=\{z\in X:d(T^jz,T^jx)<\eps,\ 0\le j<n\}.
\]
For a corresponding invariant pair $\widetilde\mu,\mu=\pi_*\widetilde\mu$,
the forward mass comparison has the form
\[
 e^{-a_n^+(\eps)}\mu(B_f(x,\eps_-,n))
 \le\widetilde\mu(B_F(\hat x,\eps,n))
 \le e^{a_n^+(\eps)}\mu(B_f(x,\eps_+,n)),
 \qquad x=\pi\hat x.
\]
For each sufficiently small fixed $\eps>0$, the radii $\eps_\pm(\eps)$
are positive, independent of time and center, and tend to zero with $\eps$.
The errors satisfy $a_n^+(\eps)/n\to0$ and are uniform in the lifted
centers on the specified Gibbs set. The backward comparison reverses the roles
of the two masses and uses prescribed lifts of base centers.
The definitions~\ref{def:forward-bowen-comparison}
and~\ref{def:backward-bowen-comparison} specify the compatible Gibbs sets.
definitions~\ref{def:sequential-forward-comparison}
and~\ref{def:sequential-backward-comparison} give the suffix estimates and
measurability requirements for prescribed sequential blocks.

The principal statements are collected below.

\begin{theorem}[\textbf{General cohomological reduction and transfer principle}]
\label{thm:master-reduction}
Assume the compact semiconjugacy and uniform fiber contraction of
Section~\ref{ass:selector}. 
\begin{enumerate}[label=\textnormal{(\roman*)},ref=\textnormal{(\roman*)}]
\item\label{item:main-normalization}
\emph{Cohomological reduction.} Suppose additionally that $\pi$ admits a
continuous section $s:M\to N$. For every $\beta$-H\"older potential
$\varphi:N\to\mathbb R$, $0<\beta\le1$, the series
\[
 u(\hat x)=\sum_{j\ge0}\bigl[\varphi(F^j\hat x)-\varphi(F^js(\pi\hat x))\bigr]
\]
converges absolutely and uniformly, defines $u\in C(N)$, and satisfies
\[
 \|u\|_\infty\le
 \frac{[\varphi]_\beta\operatorname{diam}(N)^\beta}{1-\lambda^\beta}.
\]
There exists $\psi\in C(M)$ such that
$\varphi=\psi\circ\pi+u-u\circ F$.
The additional summable-defect assumptions of
Proposition~\ref{prop:holder-reduction} give the stated H\"older regularity.

\item\label{item:main-measures}
\emph{Invariant measures and entropy.} Without assuming a continuous section
or a potential identity, projection is an affine bijection
$\pi_*:\mathcal M_F(N)\to\mathcal M_f(M)$ preserving ergodicity.
The fiberwise topological entropy is zero, and corresponding invariant
measures satisfy $h_{\widetilde\mu}(F)=h_\mu(f)$.

\item\label{item:main-thermodynamics}
\emph{Variational correspondence.} Suppose, independently of item~\ref{item:main-normalization},
that $\varphi,u\in C(N)$ and $\psi\in C(M)$ satisfy
\begin{equation}\label{eq:master-cohomology}
\varphi=\psi\circ\pi+u-u\circ F.
\end{equation}
Then corresponding invariant measures have equal potential averages, and
\[
 m(\varphi,F)=m(\psi,f),\qquad P_F(\varphi)=P_f(\psi).
\]
Minimizing measures correspond bijectively; equilibrium states correspond
whenever the common pressure is finite.

\item\label{item:main-subcohomology}
\emph{Subcohomology.} Under identity~\eqref{eq:master-cohomology},
a quotient subaction $\psi\ge v-v\circ f$, $v\in C(M)$, lifts to
$\varphi\ge U-U\circ F$ with $U=u+v\circ\pi$.
If the quotient has property $\mathrm{SC}$ for a class containing $\psi$,
then $m(\varphi,F)\ge0$ guarantees such a subaction.
The lifted $U$ is H\"older whenever $u$ and $v\circ\pi$ are H\"older.

\item\label{item:main-livsic}
\emph{Liv\v{s}ic transfer.} Each $p$ with $f^np=p$ has a unique
lift $\hat p\in N_p$ fixed by $F^n$; this periodic-lifting assertion
uses only fiber contraction. Under identity~\eqref{eq:master-cohomology},
if the quotient admits a periodic-obstruction Liv\v{s}ic theorem applicable
to $\psi$, vanishing periodic sums of $\varphi$ imply
$\psi=v-v\circ f$ and $\varphi=U-U\circ F$, with $U=u+v\circ\pi$.
The transfer function is H\"older whenever $u$ and $v\circ\pi$ are H\"older.
An alternative criterion using two-sided subcohomology and dense periodic
points appears in Corollary~\ref{cor:two-sided-sc-livsic}.

\item\label{item:main-gibbs}
\emph{Uniform weak Gibbs transfer.} Assume identity~\eqref{eq:master-cohomology}
and finite common pressure. For a corresponding invariant pair and compatible
Gibbs sets, forward subexponential Bowen-mass comparability in the sense of
Definition~\ref{def:forward-bowen-comparison} transfers uniform weak Gibbs
estimates from the quotient to the extension. Backward comparability in the
sense of Definition~\ref{def:backward-bowen-comparison} transfers them in the
reverse direction; both comparisons yield equivalence. The transferred
logarithmic error is bounded by the source Gibbs error at the comparison
radii, the $o(n)$ mass-comparison error, and $2\|u\|_\infty$.

\item\label{item:main-sequential}
\emph{Sequential Gibbs transfer.} Assume identity~\eqref{eq:master-cohomology}.
Forward or backward sequential Bowen-mass comparability under the respective
measurability, Gibbs-set, and suffix conditions of
Definitions~\ref{def:sequential-forward-comparison} and
\ref{def:sequential-backward-comparison} transfers blockwise Bowen sequential
weak Gibbs estimates in the corresponding direction, preserving the
sequential Gibbs pressure. Forward admissible sequences are pulled back
from the quotient; backward sequences are induced by a prescribed measurable
lift. If the source estimates are strong Bowen sequential Gibbs and the
comparison errors are uniformly bounded, strong estimates hold along the
same prescribed sequences. Blockwise subexponential errors are normalized
by the full admissible block length, not by suffix length. The sequential
pressure is not identified here with the full topological pressure.
\end{enumerate}
\end{theorem}

\begin{remark}[Positive-margin subactions]
The positive-margin statement does not depend on the extension: for any
continuous $T:X\to X$ on a compact metric space and $\phi\in C(X)$,
$m(\phi,T)>0$ yields a Lipschitz $w$ satisfying
$\phi\ge w-w\circ T$; see Proposition~\ref{prop:positive}
and its subsequent approximation argument.
\end{remark}

The proof is collected at the end of Section~\ref{sec:livsic}.

\subsection{A sharp renewal application}

The renewal model realizes the logarithmic loss in
Theorem~\ref{thm:master-reduction}\ref{item:main-gibbs}.

\begin{theorem}[Sharp mass comparison in the renewal extension]
\label{thm:intro-renewal}
Let $M=\{0,1\}^{\mathbb N_0}$, let $\sigma$ be the one-sided shift,
and set $N=M\times[0,1]$. Endow $M$ with
\[
 d_M(x,x')=2^{-\min\{j\ge0:x_j\ne x'_j\}}\quad(x\ne x'),
 \qquad d_M(x,x)=0,
\]
and $N$ with $d_N((x,y),(x',y'))=\max\{d_M(x,x'),|y-y'|\}$.
For every $\alpha>2$, the stationary renewal measure $\mu$ with gap law
\[
 p_k=Z^{-1}(k+1)^{-\alpha},\qquad
 Z=\sum_{k\ge0}(k+1)^{-\alpha},
\]
has a unique invariant lift under
\[
 F:N\to N,\qquad F(x,y)=\left(\sigma x,\frac y3+\frac23x_0\right),
 \qquad \pi(x,y)=x.
\]
This lift is supported on $M\times C$, where $C$ is the middle-thirds Cantor set,
and is not supported on a measurable graph. With the maximum product
metric, for every sufficiently small fixed $\eps>0$,
\[
 \frac{c_\eps}{n+C_\eps}
 \le
 \frac{\widetilde\mu(B_F(\hat x,\eps,n))}
      {\mu(B_\sigma(\pi\hat x,\eps,n))}
 \le1
 \qquad(\hat x\in M\times C,\ n\ge1).
\]
For each fixed $0<\eps<1/3$ and every sufficiently large $n$, the ratio
is at most $D_\eps/n$ on a positive $\widetilde\mu$-measure set of
centers. More precisely, for any full-measure set $G\subset M\times C$,
define the smallest same-radius forward logarithmic error by
\[
 A_n(\eps;G)=\sup_{\hat x\in G}
 \log\frac{\mu(B_\sigma(\pi\hat x,\eps,n))}
 {\widetilde\mu(B_F(\hat x,\eps,n))}.
\]
Then $A_n(\eps;G)=\log n+O_\eps(1)$ for each fixed $0<\eps<1/3$.
For any fixed positive base comparison radii, every forward comparison
error valid on $G$ still satisfies $a_n^+(\eps)\ge\log n-O(1)$.
Thus fixed changes of those radii cannot yield bounded forward error;
the matching upper estimate is attained by the same-radius comparison.
Constants in the lower obstruction may depend on the chosen base radii.
Along the prescribed section $\ell(x)=(x,0)$ one instead has
\[
 c'_\eps\mu(B_\sigma(x,\eps,n))
 \le\widetilde\mu(B_F(\ell(x),\eps,n))
 \le\mu(B_\sigma(x,\eps,n))
 \qquad(x\in M,\ n\ge1).
\]
There are continuous non-H\"older potentials related by a continuous
cohomological identity for which both measures are uniform weak Gibbs,
neither is uniform strong Gibbs, and each potential has at least two
equilibrium states.
\end{theorem}

The construction and proof are given in
Section~\ref{subsec:renewal-unbounded}, specifically
Proposition~\ref{prop:renewal-unbounded} and
Corollary~\ref{cor:renewal-equilibria}. The backward section selects centers for the comparison; the invariant
lift itself has no graph representation. Since the forward obstruction
occurs on sets of positive measure, it persists on every full-measure set. Proposition~\ref{prop:pointwise-not-uniform} addresses a
different limitation: pointwise subexponential Gibbs errors need not admit
a uniform subexponential envelope.

\section{Cohomological Reduction and Regularity}\label{sec:reduction}

This section proves item~\ref{item:main-normalization} of
Theorem~\ref{thm:master-reduction} and isolates the additional input needed
for H\"older regularity.

\begin{proposition}[Cohomological reduction]
\label{thm:reduction}
Assume the compact semiconjugacy and uniform fiber contraction hypotheses
of Section~\ref{ass:selector}. Suppose, in addition, that $\pi$ admits a
continuous section $s:M\to N$, and set $r=s\circ\pi$.
Let $\varphi:N\to\R$ be $\beta$-H\"older. Define
\begin{equation}
 u(\hat x)=\sum_{j=0}^{\infty}
 \bigl[\varphi(F^j\hat x)-\varphi(F^jr(\hat x))\bigr].
 \label{eq:u-series-main}
\end{equation}
The series converges absolutely and uniformly, defining a continuous
function $u:N\to\R$. Moreover,
\begin{equation}
 \|u\|_\infty\le
 \frac{[\varphi]_\beta\diam(N)^\beta}{1-\lambda^\beta}.
 \label{eq:u-sup-bound}
\end{equation}
The function
\[
 \bar\varphi:=\varphi-u+u\circ F
\]
is constant on each fiber. Hence there is a unique continuous
$\psi:M\to\R$ such that
\begin{equation}
 \boxed{\varphi=\psi\circ\pi+u-u\circ F.}
 \label{eq:main-cohomology}
\end{equation}
H\"older regularity is not asserted from fiber contraction alone; a
quantitative sufficient criterion is given in
Proposition~\ref{prop:holder-reduction}.
\end{proposition}

\begin{proof}
Since $\hat x$ and $r(\hat x)$ lie in the same fiber,
\[
d_N(F^j\hat x,F^jr(\hat x))
\le \lambda^j d_N(\hat x,r(\hat x)).
\]
Hence
\begin{align*}
|\varphi(F^j\hat x)-\varphi(F^jr(\hat x))|
&\le [\varphi]_\beta
 d_N(F^j\hat x,F^jr(\hat x))^\beta\\
&\le [\varphi]_\beta\lambda^{\beta j}
 d_N(\hat x,r(\hat x))^\beta\\
&\le [\varphi]_\beta\lambda^{\beta j}\diam(N)^\beta.
\end{align*}
The Weierstrass $M$-test yields absolute uniform convergence and
\eqref{eq:u-sup-bound}; continuity follows because the partial sums are
continuous.

We next show that
$\bar\varphi=\varphi-u+u\circ F$ is constant on fibers. Let
$\pi(\hat x)=\pi(\hat y)$. Since $r(\hat x)=r(\hat y)$, subtraction of
the two absolutely convergent series gives
\begin{align}
u(\hat x)-u(\hat y)
&=\sum_{j=0}^{\infty}
\bigl[\varphi(F^j\hat x)-\varphi(F^j\hat y)\bigr].
\label{eq:u-fiber-difference}
\end{align}
The points $F\hat x$ and $F\hat y$ also lie in the same fiber, and hence
\begin{align}
u(F\hat x)-u(F\hat y)
&=\sum_{j=0}^{\infty}
\bigl[\varphi(F^{j+1}\hat x)-\varphi(F^{j+1}\hat y)\bigr]\\
&=\sum_{j=1}^{\infty}
\bigl[\varphi(F^j\hat x)-\varphi(F^j\hat y)\bigr].
\label{eq:u-fiber-difference-shifted}
\end{align}
Both series are absolutely convergent because
\[
|\varphi(F^j\hat x)-\varphi(F^j\hat y)|
\le [\varphi]_\beta\lambda^{\beta j}
 d_N(\hat x,\hat y)^\beta.
\]
Subtracting \eqref{eq:u-fiber-difference} and
\eqref{eq:u-fiber-difference-shifted} therefore yields
\begin{align*}
\bar\varphi(\hat x)-\bar\varphi(\hat y)
&=\varphi(\hat x)-\varphi(\hat y)
 -\bigl(u(\hat x)-u(\hat y)\bigr)
 +\bigl(u(F\hat x)-u(F\hat y)\bigr)\\
&=0.
\end{align*}
Thus $\bar\varphi$ is constant on every fiber. Define
\[
\psi(x)=\bar\varphi(s(x)).
\]
Then $\psi$ is continuous and $\bar\varphi=\psi\circ\pi$. For the
canonical transfer function $u$ fixed by \eqref{eq:u-series-main}, the quotient
potential $\psi$ is unique by surjectivity of $\pi$.
\end{proof}

\begin{remark}
\label{rem:gauge-freedom}
The quotient potential is canonical only after fixing the transfer function
$u$ by the stable-difference series \eqref{eq:u-series-main}. Indeed, if
$w\in C(M)$ and
\[
 u'=u+w\circ\pi,
 \qquad
 \psi'=\psi-w+w\circ f,
\]
then semiconjugacy gives the equally valid decomposition
\[
 \varphi=\psi'\circ\pi+u'-u'\circ F.
\]
Thus arbitrary cohomological decompositions have the natural gauge freedom
$(u,\psi)\mapsto(u+w\circ\pi,\psi-w+w\circ f)$. The phrase
``the quotient potential'' below refers to the canonical representative
associated with \eqref{eq:u-series-main}.
\end{remark}

\begin{remark}
For $F(x,y)=(f(x),ay+\rho(x))$, $|a|<1$, on an invariant compact
cylinder, the section $s(x)=(x,0)$ and the potential
$\varphi(x,y)=g(x)+by$ give directly
\[
 u(x,y)=\frac{b}{1-a}y,
 \qquad \psi(x)=g(x)+\frac{b}{1-a}\rho(x).
\]
The fiber term is converted into a quotient contribution rather than
omitted. This formula also verifies the regularity of the reduced potential
directly; see Proposition~\ref{prop:affine-explicit}.
\end{remark}

\begin{proposition}[Summable-defect H\"older criterion]
\label{prop:holder-reduction}
For $j\ge0$ define the fiber defect
\begin{equation}
\Delta_j(\hat x)
:=\varphi(F^j\hat x)-\varphi(F^jr(\hat x)).
\label{eq:defect-delta}
\end{equation}
Assume that, for some $0<\alpha\le1$, each $\Delta_j$ is
$\alpha$-H\"older and that
\begin{equation}
\sum_{j=0}^{\infty}[\Delta_j]_\alpha<\infty.
\label{eq:defect-holder-sum}
\end{equation}
Then the transfer function $u=\sum_{j\ge0}\Delta_j$ is
$\alpha$-H\"older and
\begin{equation}
[u]_\alpha\le\sum_{j=0}^{\infty}[\Delta_j]_\alpha.
\label{eq:u-holder-bound}
\end{equation}
If, in addition, $F$ is Lipschitz, $\alpha\le\beta$, and the section
$s:M\to N$ is $\sigma$-H\"older for some $0<\sigma\le1$, then $\bar\varphi$ is
$\alpha$-H\"older and $\psi=\bar\varphi\circ s$ is
$\alpha\sigma$-H\"older.
\end{proposition}

\begin{proof}
For $\hat x,\hat y\in N$, absolute convergence and
\eqref{eq:defect-holder-sum} give
\begin{align*}
|u(\hat x)-u(\hat y)|
&\le\sum_{j=0}^{\infty}
|\Delta_j(\hat x)-\Delta_j(\hat y)|\\
&\le\left(\sum_{j=0}^{\infty}[\Delta_j]_\alpha\right)
 d_N(\hat x,\hat y)^\alpha.
\end{align*}
This proves \eqref{eq:u-holder-bound}. If $F$ is Lipschitz, then
$u\circ F$ is $\alpha$-H\"older. Since $\alpha\le\beta$ and $N$ is
compact, every $\beta$-H\"older function is also $\alpha$-H\"older;
therefore
$\bar\varphi=\varphi-u+u\circ F$ is $\alpha$-H\"older. Finally, writing
\[
[s]_\sigma
:=\sup_{x\neq y}\frac{d_N(s(x),s(y))}{d_M(x,y)^\sigma},
\]
we obtain
\[
|\psi(x)-\psi(y)|
=|\bar\varphi(s(x))-\bar\varphi(s(y))|
\le [\bar\varphi]_\alpha [s]_\sigma^\alpha
 d_M(x,y)^{\alpha\sigma}.
\]
\end{proof}

\begin{remark}
In hyperbolic, expanding, or zooming models, condition
\eqref{eq:defect-holder-sum} may follow from holonomy, inverse-branch,
or distortion estimates.
\end{remark}

\section{Invariant Measures, Minimizing Measures and Ergodic Optimization}\label{sec:measures}

\begin{theorem}[Bijectivity of invariant lifts]
\label{thm:lift-bijection}
Under the compact semiconjugacy and uniform fiber contraction (without
requiring a continuous section), the projection
\[
\boxed{\pi_*:\M_F(N)\longrightarrow\M_f(M)}
\]
is a bijection. Compactness and surjectivity supply invariant lifts,
while uniform fiber contraction makes each invariant lift unique. In
particular, every $f$-invariant probability measure has a unique
$F$-invariant lift; no ergodicity assumption is required.
\end{theorem}

\begin{proof}
Let $\mu\in\M_f(M)$. Since $\pi$ is a continuous surjection between compact metric spaces,
it admits a Borel measurable section $s_0:M\to N$ (by the measurable
selection theorem). Thus the probability $\nu_0=(s_0)_*\mu$ satisfies
$\pi_*\nu_0=\mu$. Set
\[
\nu_n=\frac1n\sum_{j=0}^{n-1}F_*^j\nu_0.
\]
Every weak-$*$ accumulation point $\widetilde\mu$ of $(\nu_n)$ is
$F$-invariant by the Krylov--Bogolyubov argument. Moreover,
\[
\pi_*F_*^j\nu_0=f_*^j\pi_*\nu_0=f_*^j\mu=\mu,
\]
so $\pi_*\widetilde\mu=\mu$. Thus projection is surjective.

To prove injectivity, suppose
$\pi_*\widetilde\mu_1=\pi_*\widetilde\mu_2=\mu$. Disintegrate over the
common factor:
\[
\widetilde\mu_i=\int_M\widetilde\mu_{i,x}\,d\mu(x),
\qquad \widetilde\mu_{i,x}(N_x)=1
\quad\text{for $\mu$-a.e. }x.
\]
Use the relatively independent joining
\begin{equation}
\kappa=\int_M
(\widetilde\mu_{1,x}\otimes\widetilde\mu_{2,x})\,d\mu(x).
\label{eq:relative-product}
\end{equation}
The coupling has marginals $\widetilde\mu_1$ and $\widetilde\mu_2$ and is supported on
pairs in the same fiber. Since the marginals are invariant,
$(F^n\times F^n)_*\kappa$ is again a coupling of the same two measures.
Therefore
\begin{align*}
W_1(\widetilde\mu_1,\widetilde\mu_2)
&\le \int d_N(F^n\hat x,F^n\hat y)\,d\kappa(\hat x,\hat y)\\
&\le \lambda^n\int d_N(\hat x,\hat y)\,d\kappa\\
&\le \lambda^n\diam(N).
\end{align*}
Taking $n\to\infty$ gives $W_1(\widetilde\mu_1,\widetilde\mu_2)=0$,
and the measures coincide.
\end{proof}

\begin{corollary}
\label{cor:average-correspondence}
If $\varphi=\psi\circ\pi+u-u\circ F$ and
$\widetilde\mu\leftrightarrow\mu$ under the bijection above, then
\begin{equation}
\int_N\varphi\,d\widetilde\mu=\int_M\psi\,d\mu.
\label{eq:average-correspondence-v4}
\end{equation}
Consequently $m(\varphi,F)=m(\psi,f)$.
\end{corollary}

\begin{proof}
Invariance gives
$\int(u-u\circ F)\,d\widetilde\mu=0$, while
$\int\psi\circ\pi\,d\widetilde\mu=\int\psi\,d\mu$. The equality of
minimum values follows because $\pi_*$ is bijective in invariant measures.
\end{proof}

Define
\[
\Lambda(\varphi,F)=\left\{\widetilde\mu\in\M_F(N):
\int\varphi\,d\widetilde\mu=m(\varphi,F)\right\},
\]
and define $\Lambda(\psi,f)$ analogously.

\begin{theorem}[Reduction of ergodic optimization]
\label{thm:ergodic-optimization}
The projection restricts to a bijection
\[
\boxed{\pi_*:\Lambda(\varphi,F)\longrightarrow\Lambda(\psi,f).}
\]
Thus, uniqueness or non-uniqueness of minimizing measures, and ergodicity of
the corresponding minimizing measure, are preserved by the reduction.
\end{theorem}

\begin{proof}
For corresponding invariant measures the objective functions agree by
\eqref{eq:average-correspondence-v4}, and their minimum values agree by the
preceding corollary. Hence one measure is minimizing if and only if its
projection/lift is minimizing. Bijectivity follows from
Theorem~\ref{thm:lift-bijection}. It remains only to justify the statement
about ergodicity. The map $\pi_*$ is affine. Its inverse is affine as well:
if $\mu=t\mu_1+(1-t)\mu_2$ and $\widetilde\mu_i$ is the unique invariant
lift of $\mu_i$, then $t\widetilde\mu_1+(1-t)\widetilde\mu_2$ is an invariant
lift of $\mu$, hence equals the unique lift of $\mu$. Thus $\pi_*$ is an
affine bijection between the compact convex sets $\M_F(N)$ and $\M_f(M)$.
Affine bijections preserve extreme points, and the extreme points of the
invariant-measure simplex are precisely the ergodic measures. Hence
ergodicity is preserved in both directions.
\end{proof}

\section{Relative Entropy, Pressure and Equilibrium States}\label{sec:pressure}

We first fix the pressure convention used below. For a continuous map
$T:X\to X$ of a compact metric space and a continuous potential $\phi$, set
\begin{equation}
P_T(\phi)
:=\sup_{\mu\in\M_T}
\left(h_\mu(T)+\int\phi\,d\mu\right).
\label{eq:variational-pressure-definition}
\end{equation}
By the variational principle for continuous maps of compact metric spaces,
this quantity agrees with the usual topological pressure defined via
separated (equivalently, spanning) sets; see, for instance, Bowen, Ruelle and Walters
\cite{Bowen,Ruelle,Walters}. Thus no distinction between variational and topological
pressure will be needed below.
When equilibrium states are discussed, we work in the standard finite-pressure
regime; this includes all concrete applications considered below.

For a compact set $K\subset N$, let $s_n(\eps,K)$ denote the maximal
cardinality of an $(n,\eps)$-separated subset of $K$ for the Bowen metric
$d_n^F$. We use the fiberwise relative topological entropy
\begin{equation}
h_{\mathrm{top}}(F\mid\pi)
:=\sup_{x\in M} h_{\mathrm{top}}(F,N_x),
\qquad N_x:=\pi^{-1}(x),
\label{eq:relative-topological-entropy-definition}
\end{equation}
where
\[
h_{\mathrm{top}}(F,K)
:=\lim_{\eps\downarrow0}\limsup_{n\to\infty}
\frac1n\log s_n(\eps,K).
\]
This is the fiberwise entropy quantity entering the
Ledrappier--Walters relative variational inequality used below. We use this
explicit fiberwise formulation throughout; no stronger relative-entropy
notion is required for the pressure argument.

\begin{lemma}
\label{lem:relative-zero}
Under uniform fiber contraction,
\[
\boxed{h_{\mathrm{top}}(F\mid\pi)=0.}
\]
More precisely, for every $\eps>0$, every $x\in M$, and every $n\ge1$,
the maximal cardinality $s_n(\eps,N_x)$ of an $(n,\eps)$-separated subset
of $N_x$ satisfies
\begin{equation}
 s_n(\eps,N_x)\le s_1(\eps,N),
 \label{eq:relative-separated-bound}
\end{equation}
where the right-hand side is finite and independent of $x$ and $n$.
\end{lemma}

\begin{proof}
For $\hat x,\hat y\in N_x$ and $j\ge0$,
\[
d_N(F^j\hat x,F^j\hat y)\le\lambda^j d_N(\hat x,\hat y)
\le d_N(\hat x,\hat y).
\]
Hence, the Bowen metric on a fiber is exactly
\[
d_n^F(\hat x,\hat y)
:=\max_{0\le j<n}d_N(F^j\hat x,F^j\hat y)
=d_N(\hat x,\hat y),
\]
because the term $j=0$ occurs in the maximum. Thus, every
$(n,\eps)$-separated subset of $N_x$ is already $\eps$-separated for the
original metric. Since $N$ is compact, $s_1(\eps,N)<\infty$, proving
\eqref{eq:relative-separated-bound}. Therefore, uniformly in $x$,
\[
\limsup_{n\to\infty}\frac1n\log s_n(\eps,N_x)=0.
\]
Letting $\eps\downarrow0$ and then taking the supremum over $x$ proves
\eqref{eq:relative-topological-entropy-definition} is zero.
\end{proof}

\begin{lemma}
\label{lem:LW-relative}
For a compact semiconjugacy
\[
\pi\circ F=f\circ\pi,
\]
fix $\mu\in\M_f(M)$. The Ledrappier--Walters relative variational principle
gives the relative entropy bound
\begin{equation}
\sup_{\substack{\widetilde\nu\in\M_F(N)\\ \pi_*\widetilde\nu=\mu}}
 h_{\widetilde\nu}(F)
\le
h_\mu(f)+\int_M h_{\mathrm{top}}(F,N_x)\,d\mu(x).
\label{eq:LW-exact-relative}
\end{equation}
Consequently, for every invariant lift $\widetilde\mu$ of $\mu$,
\[
h_\mu(f)\le h_{\widetilde\mu}(F)
\le h_\mu(f)+h_{\mathrm{top}}(F\mid\pi).
\]
By Lemma~\ref{lem:relative-zero}, every invariant lift therefore satisfies
\[
h_{\widetilde\mu}(F)=h_\mu(f).
\]
\end{lemma}

\begin{proof}
The inequality \eqref{eq:LW-exact-relative} follows from the relative variational principle of
Ledrappier--Walters~\cite{LedrappierWalters}. The lower bound
$h_\mu(f)\le h_{\widetilde\mu}(F)$ is the monotonicity of the metric entropy under
factors. For the upper bound, every particular lift is bounded by the
supremum in \eqref{eq:LW-exact-relative}, and
\[
\int_M h_{\mathrm{top}}(F,N_x)\,d\mu(x)
\le h_{\mathrm{top}}(F\mid\pi).
\]
The Lemma~\ref{lem:relative-zero} then gives the equality of the two metric entropies.
\end{proof}

\begin{theorem}
\label{thm:pressure}
Let $\widetilde\mu\leftrightarrow\mu$ be the corresponding invariant measures.
Then
\begin{equation}
\boxed{h_{\widetilde\mu}(F)=h_\mu(f).}
\label{eq:entropy-equality}
\end{equation}
If $\varphi=\psi\circ\pi+u-u\circ F$, then
\begin{equation}
\boxed{P_F(\varphi)=P_f(\psi).}
\label{eq:pressure-equality}
\end{equation}
If, in addition, the common pressure $P_F(\varphi)=P_f(\psi)$ is
finite, then $\widetilde\mu$ is an equilibrium state for $(F,\varphi)$
if and only if $\mu$ is an equilibrium state for $(f,\psi)$.
\end{theorem}

\begin{proof}
Lemma~\ref{lem:LW-relative}, together with
Lemma~\ref{lem:relative-zero}, gives
\eqref{eq:entropy-equality}. By Corollary~\ref{cor:average-correspondence},
\[
\int_N\varphi\,d\widetilde\mu=\int_M\psi\,d\mu.
\]
Hence, the free-energy functionals agree for every corresponding pair:
\begin{equation}
h_{\widetilde\mu}(F)+\int_N\varphi\,d\widetilde\mu
=
h_\mu(f)+\int_M\psi\,d\mu.
\label{eq:free-energy-correspondence}
\end{equation}
Since Theorem~\ref{thm:lift-bijection} gives a bijection
$\pi_*:\M_F(N)\to\M_f(M)$, taking the suprema in
\eqref{eq:free-energy-correspondence} yields
\[
P_F(\varphi)=P_f(\psi).
\]
When the common pressure is finite, the same pointwise identity of
free-energy functionals, together with the bijection of invariant measures,
shows that a measure attains the pressure supremum on one side if and only
if its corresponding measure attains the supremum on the other side.
The equilibrium-state correspondence follows.
\end{proof}

\section{Weak Gibbs Transfer under Quantitative Bowen-Mass Comparison}\label{sec:weak-gibbs}

We prove item~\ref{item:main-gibbs} of
Theorem~\ref{thm:master-reduction}. Throughout the transfer arguments, $\varphi,\psi,u$ are continuous and
satisfy $\varphi=\psi\circ\pi+u-u\circ F$. The identity may be obtained from Proposition~\ref{thm:reduction}
or verified directly.

For $T:X\to X$ write
\[
B_T(x,\eps,n)=\{y:d(T^jy,T^jx)<\eps,\ 0\le j<n\}.
\]
Throughout this section we assume that the pressures appearing in the weak
Gibbs expressions are finite. Let $G_T\subset X$ be a prescribed Gibbs set.
We use the uniform-on-$G_T$ convention for weak Gibbs estimates: the
subexponential constants below depend on the radius and on $n$, but not on
the center $x\in G_T$. We say that a probability measure $\nu$ is
\emph{uniform weak Gibbs for $(T,\phi)$ on $G_T$} if for every
sufficiently small $\eps>0$ there are numbers $K_n(\eps)\ge1$ such that
\begin{equation}
\lim_{n\to\infty}\frac1n\log K_n(\eps)=0
\label{eq:weak-gibbs-subexp}
\end{equation}
and, for every $x\in G_T$ and every $n\ge1$,
\begin{equation}
K_n(\eps)^{-1}
\le
\frac{\nu(B_T(x,\eps,n))}
{\exp(S_n\phi(x)-nP_T(\phi))}
\le K_n(\eps).
\label{eq:weak-gibbs-def}
\end{equation}
We abbreviate this uniform-on-$G_T$ property simply as ``weak Gibbs'' in the
statements below. This convention is stronger than pointwise variants in
which the subexponential error is allowed to depend on the center.

We use subexponential comparisons of Bowen-ball masses for a fixed
pair of corresponding invariant measures. The estimates below are
uniform over the specified Gibbs sets.

\begin{definition}
\label{def:forward-bowen-comparison}
Let $\widetilde\mu\in\M_F(N)$ and $\mu=\pi_*\widetilde\mu$. Let
$G_F\subset N$ and $G_f\subset M$ satisfy $\pi(G_F)\subset G_f$. We say
that $(\widetilde\mu,\mu)$ satisfies \emph{forward subexponential
Bowen-mass comparability on $(G_F,G_f)$} if, for every sufficiently small
$\eps>0$, there exist radii
$\eps_-(\eps),\eps_+(\eps)>0$ tending to zero with $\eps$ and numbers
$a_n^{+}(\eps)\ge0$ such that
\[
\lim_{n\to\infty}\frac{a_n^{+}(\eps)}n=0
\]
and, for every $\hat x\in G_F$, $x=\pi(\hat x)$, and every $n\ge1$,
\begin{equation}
e^{-a_n^{+}(\eps)}\mu(B_f(x,\eps_-,n))
\le \widetilde\mu(B_F(\hat x,\eps,n))
\le e^{a_n^{+}(\eps)}\mu(B_f(x,\eps_+,n)).
\label{eq:bowen-comp-forward}
\end{equation}
\end{definition}

\begin{definition}
\label{def:backward-bowen-comparison}
Let $\widetilde\mu\in\M_F(N)$ and $\mu=\pi_*\widetilde\mu$. Let
$G_F\subset N$ and $G_f\subset M$ be such that every $x\in G_f$ admits
at least one lift $\hat x\in G_F$ with $\pi(\hat x)=x$. We say that
$(\widetilde\mu,\mu)$ satisfies \emph{backward subexponential Bowen-mass
comparability on $(G_F,G_f)$} if, for every sufficiently small $\eps>0$,
there exist radii $\eps_-(\eps),\eps_+(\eps)>0$ tending to zero with
$\eps$ and numbers $a_n^{-}(\eps)\ge0$ such that
\[
\lim_{n\to\infty}\frac{a_n^{-}(\eps)}n=0
\]
and, for every $x\in G_f$, for a prescribed lift $\hat x\in G_F$ with
$\pi(\hat x)=x$, and for every $n\ge1$,
\begin{equation}
e^{-a_n^{-}(\eps)}\widetilde\mu(B_F(\hat x,\eps_-,n))
\le \mu(B_f(x,\eps,n))
\le e^{a_n^{-}(\eps)}\widetilde\mu(B_F(\hat x,\eps_+,n)).
\label{eq:bowen-comp-backward}
\end{equation}
\end{definition}

The logarithmic formulation above makes explicit that only an $o(n)$ error
is required in the Bowen-ball masses. Similarly,  one may write
$A_n^{\pm}(\eps)=e^{a_n^{\pm}(\eps)}$ with
$n^{-1}\log A_n^{\pm}(\eps)\to0$.

\begin{lemma}
\label{lem:radius-reparametrization}
Let $\nu$ be weak Gibbs for $(T,\phi)$ in the sense of
\eqref{eq:weak-gibbs-def}. Suppose that for every sufficiently small
$\eps>0$, two radii $\delta_-(\eps),\delta_+(\eps)>0$ satisfy
\[
\delta_\pm(\eps)\longrightarrow0\qquad(\eps\to0).
\]
Then, the weak Gibbs estimates can be evaluated at
$\delta_-(\eps)$ and $\delta_+(\eps)$ without changing the subexponential
character of the Gibbs constants. More precisely, if
$K_n(\delta_\pm(\eps))$ are the corresponding Gibbs constants, then
\[
\frac1n\log
\max\{K_n(\delta_-(\eps)),K_n(\delta_+(\eps))\}\longrightarrow0.
\]
The same conclusion remains valid after finitely many further radius
reparameterizations, provided all resulting radii tend to zero with the
target radius.
\end{lemma}

\begin{proof}
For each  sufficiently small fixed target radius $\eps$, the radii
$\delta_\pm(\eps)$ are themselves fixed positive radii. Hence, the weak Gibbs
definition applies to each of them, and both corresponding logarithmic
constants are $o(n)$. The maximum of finitely many $o(n)$ sequences is again
$o(n)$.
\end{proof}

\begin{theorem}
\label{thm:weak-gibbs-forward}
Assume
\[
\varphi=\psi\circ\pi+u-u\circ F,
\]
and let $\widetilde\mu\leftrightarrow\mu$ be the corresponding invariant
measures. Let $G_F\subset N$ and $G_f\subset M$ satisfy
$\pi(G_F)\subset G_f$. If $(\widetilde\mu,\mu)$ satisfies the forward
subexponential Bowen-mass comparability on $(G_F,G_f)$ and $\mu$ is weak
Gibbs for $(f,\psi)$ on $G_f$, then $\widetilde\mu$ is weak Gibbs for
$(F,\varphi)$ on $G_F$.
\end{theorem}

\begin{proof}
For every $\hat x$ and every $n$,
\begin{equation}
S_n\varphi(\hat x)
=S_n\psi(\pi\hat x)+u(\hat x)-u(F^n\hat x),
\label{eq:birkhoff-coboundary}
\end{equation}
so
\begin{equation}
\bigl|S_n\varphi(\hat x)-S_n\psi(\pi\hat x)\bigr|
\le2\|u\|_\infty.
\label{eq:birkhoff-bounded-error}
\end{equation}
Pressure equality follows from Theorem~\ref{thm:pressure}. Take $\hat x\in G_F$ and set
$x=\pi(\hat x)\in G_f$. By \eqref{eq:bowen-comp-forward} and the weak Gibbs
bounds for $\mu$ at the radii $\eps_-$ and $\eps_+$,
\begin{align*}
\widetilde\mu(B_F(\hat x,\eps,n))
&\le e^{a_n^+(\eps)}K_n(\eps_+)
 \exp(S_n\psi(x)-nP_f(\psi))\\
&\le e^{a_n^+(\eps)}K_n(\eps_+)e^{2\|u\|_\infty}
 \exp(S_n\varphi(\hat x)-nP_F(\varphi)),
\end{align*}
and analogously
\begin{align*}
\widetilde\mu(B_F(\hat x,\eps,n))
&\ge e^{-a_n^+(\eps)}K_n(\eps_-)^{-1}
 \exp(S_n\psi(x)-nP_f(\psi))\\
&\ge e^{-a_n^+(\eps)}K_n(\eps_-)^{-1}e^{-2\|u\|_\infty}
 \exp(S_n\varphi(\hat x)-nP_F(\varphi)).
\end{align*}
One may therefore choose
\[
\widetilde K_n(\eps)
=e^{a_n^+(\eps)+2\|u\|_\infty}
\max\{K_n(\eps_-),K_n(\eps_+)\}.
\]
Then $n^{-1}\log\widetilde K_n(\eps)\to0$, as required.
\end{proof}

\begin{theorem}[Backward weak Gibbs transfer]
\label{thm:weak-gibbs-backward}
Assume
$\varphi=\psi\circ\pi+u-u\circ F$, and let
$\widetilde\mu\leftrightarrow\mu$ be corresponding invariant measures.
Let $G_F\subset N$ and $G_f\subset M$ be such that every $x\in G_f$ has
a prescribed lift $\hat x\in G_F$. If $(\widetilde\mu,\mu)$ satisfies
backward subexponential Bowen-mass comparability on $(G_F,G_f)$ and
$\widetilde\mu$ is weak Gibbs for $(F,\varphi)$ on $G_F$, then $\mu$ is
weak Gibbs for $(f,\psi)$ on $G_f$.
\end{theorem}

\begin{proof}
Fix $x\in G_f$ and its prescribed lift $\hat x\in G_F$. Apply the weak
Gibbs bounds for $\widetilde\mu$ at the radii $\eps_-$ and $\eps_+$ and
then use \eqref{eq:bowen-comp-backward}. Combining these inequalities with
\eqref{eq:birkhoff-bounded-error} and
$P_F(\varphi)=P_f(\psi)$ yields
\begin{align*}
\mu(B_f(x,\eps,n))
&\le e^{a_n^-(\eps)}K_n(\eps_+)e^{2\|u\|_\infty}
\exp(S_n\psi(x)-nP_f(\psi)),\\
\mu(B_f(x,\eps,n))
&\ge e^{-a_n^-(\eps)}K_n(\eps_-)^{-1}e^{-2\|u\|_\infty}
\exp(S_n\psi(x)-nP_f(\psi)).
\end{align*}
This constant is subexponential because
$a_n^-(\eps)=o(n)$ and \eqref{eq:weak-gibbs-subexp} hold for
$\widetilde\mu$.
\end{proof}

\begin{corollary}
\label{thm:weak-gibbs}
Assume $\varphi=\psi\circ\pi+u-u\circ F$. Let
$\widetilde\mu\leftrightarrow\mu$, and let $G_F,G_f$ be Gibbs sets for
which both forward and backward subexponential Bowen-mass comparability
hold. Then
\[
\mu\text{ is weak Gibbs for }(f,\psi)\text{ on }G_f
\quad\Longleftrightarrow\quad
\widetilde\mu\text{ is weak Gibbs for }(F,\varphi)\text{ on }G_F.
\]
\end{corollary}

\begin{proof}
The forward implication is Theorem~\ref{thm:weak-gibbs-forward}, and the
reverse implication is Theorem~\ref{thm:weak-gibbs-backward}. 
\end{proof}

\begin{remark}
\label{rem:bowen-comparison-not-automatic}
The subexponential Bowen-mass comparison above is strictly stronger than
zero relative topological entropy and is not implied by the pressure
identity. The uniform continuity of $\pi$ gives one geometric inclusion for
Bowen balls: for every $\eta>0$ there exists $\eps>0$ such that
\[
\pi(B_F(\hat x,\eps,n))\subset B_f(\pi\hat x,\eta,n)
\quad\text{for every }\hat x\in N\text{ and }n\ge1.
\]
Consequently,
\[
\widetilde\mu(B_F(\hat x,\eps,n))
\le \mu(B_f(\pi\hat x,\eta,n)).
\]
For the opposite inequality, one needs control of mass across fibers;
the criteria below provide such bounds.
\end{remark}

\subsection{Geometric and measure-theoretic comparison criteria}

We give three criteria combining Bowen-ball geometry with lower bounds
on mass in the contracting direction.

Whenever radius moduli are introduced below, they are chosen positive
and nondecreasing near zero, and tend to zero with their argument.
Every inverse radius choice is made so that its parameter also tends to
zero with the target radius. For clarity, given a lower modulus $\delta_-$,
fix $t_0>0$ in its domain and put $t_k=2^{-k}t_0$. For sufficiently small
$\eta>0$, choose the largest $k=k(\eta)$ with
$\eta\le\delta_-(t_k)$. This index is finite and tends to infinity as
$\eta\downarrow0$. Thus $t_{k(\eta)}\to0$ and
$\eta\le\delta_-(t_{k(\eta)})$.
For upper moduli, choose a sufficiently small parameter, also bounded by
$\eta$, whose image is at most $\eta$. These conventions justify the
radius reparametrizations in both directions.

\begin{proposition}
\label{prop:invariant-graph-bowen}
Assume that there exists a continuous invariant section
\[
s:M\longrightarrow N,
\qquad \pi\circ s=\mathrm{id}_M,
\qquad F\circ s=s\circ f.
\]
Let $\mu\in\M_f(M)$ and let $\widetilde\mu$ be its unique invariant lift.
Then $\widetilde\mu=s_*\mu$. If $G_f\subset M$ and $G_F=s(G_f)$, then
$(\widetilde\mu,\mu)$ satisfies both forward and backward subexponential
Bowen-mass comparability on $(G_F,G_f)$ with
\[
a_n^+(\eps)=a_n^-(\eps)=0.
\]
More precisely, for every sufficiently small $\eps>0$ there are radii
$\delta_-(\eps),\delta_+(\eps)>0$, tending to zero with $\eps$, such that for
every $x\in G_f$, every $n\ge1$, and $\hat x=s(x)$,
\[
\mu(B_f(x,\delta_-,n))
\le
\widetilde\mu(B_F(\hat x,\eps,n))
\le
\mu(B_f(x,\delta_+,n)).
\]
\end{proposition}

\begin{proof}
Since $\mu$ is $f$-invariant and $F\circ s=s\circ f$,
\[
F_*s_*\mu=s_*f_*\mu=s_*\mu.
\]
Moreover $\pi_*s_*\mu=\mu$. By Theorem~\ref{thm:lift-bijection}, the invariant
lift is unique, hence $\widetilde\mu=s_*\mu$.

Fix $\eps>0$. Since $M$ is compact and $s$ is continuous, $s$ is uniformly
continuous. Choose $\delta_->0$ such that
\[
d_M(x,y)<\delta_-
\quad\Longrightarrow\quad
d_N(s(x),s(y))<\eps.
\]
If $y\in B_f(x,\delta_-,n)$, then for every $0\le j<n$,
\[
d_N(F^js(x),F^js(y))
=d_N(s(f^jx),s(f^jy))<\eps,
\]
so
\[
s(B_f(x,\delta_-,n))\subset B_F(s(x),\eps,n)\cap s(M).
\]
In contrast, $\pi|_{s(M)}:s(M)\to M$ is continuous and $s(M)$ is compact;
therefore, it is uniformly continuous. Choose $\delta_+>0$ so that
\[
d_N(s(x),s(y))<\eps
\quad\Longrightarrow\quad
d_M(x,y)<\delta_+.
\]
Then
\[
B_F(s(x),\eps,n)\cap s(M)
\subset s(B_f(x,\delta_+,n)).
\]
Since $\widetilde\mu=s_*\mu$ is supported on $s(M)$, taking measures gives the
claimed inequalities. Forward comparability follows immediately. The
backward comparability follows by applying the same inclusions after reparameterizing
the radii with the uniform-continuity moduli of $s$ and
$\pi|_{s(M)}=s^{-1}$. No multiplicative mass error is introduced, so the
logarithmic errors are identically zero.
\end{proof}

For skew products, the geometric estimate follows from contraction and
transverse continuity. The next criterion supplies the remaining mass bound.

\begin{proposition}[Uniform skew-product criterion]
\label{prop:uniform-skew-bowen}
Assume that $N=M\times Y$ is endowed with a metric uniformly equivalent to a
product metric and
\[
F(x,y)=\bigl(f(x),G_x(y)\bigr).
\]
Assume uniform fiber contraction: there is $0<\lambda<1$ such that
\[
d_Y(G_x(y),G_x(y'))\le\lambda d_Y(y,y')
\]
for all admissible points. Assume also uniform transverse continuity: there is
a nondecreasing modulus $\omega:[0,\infty)\to[0,\infty)$,
with $\omega(0)=0$ and $\omega(t)\to0$ as $t\to0$, and
\[
d_Y(G_x(y),G_{x'}(y))\le \omega(d_M(x,x'))
\]
for all admissible points.

Let
\[
\widetilde\mu=\int_M\widetilde\mu_x\,d\mu(x)
\]
be the disintegration over $\pi(x,y)=x$. Let $G_F\subset N$ be a measurable set and let
$G_f=\pi(G_F)$ be measurable. Fix a Borel version of the disintegration
$\{\widetilde\mu_x\}$; the following integrated condition is required
for every center in $G_F$ (not merely almost every center). Assume the following \emph{averaged subexponential fiber
nondegeneracy}: for every sufficiently small fiber radius $r>0$ and every
sufficiently small base Bowen radius $\delta>0$ there are numbers
$q_n(r,\delta)\ge0$ such that
\[
\frac{q_n(r,\delta)}{n}\longrightarrow0
\qquad (n\to\infty),
\]
and, uniformly for $(x,y)\in G_F$ and every $n\ge1$,
\begin{equation}
\label{eq:averaged-lower-fiber-ball-masses}
\int_{B_f(x,\delta,n)}
\widetilde\mu_{x'}(B_Y(y,r))\,d\mu(x')
\ge e^{-q_n(r,\delta)}\,\mu(B_f(x,\delta,n)).
\end{equation}
Then $(\widetilde\mu,\mu)$ satisfies forward and backward subexponential
Bowen-mass comparability on $(G_F,G_f)$, with the backward comparison
using the same chosen lifted center for each base point.
\end{proposition}

\begin{proof}
Fix a sufficiently small Bowen radius $\eps>0$, and write
$d_\times((x,y),(x',y'))=\max\{d_M(x,x'),d_Y(y,y')\}$.
To make the metric comparison explicit, uniform equivalence provides
nondecreasing moduli $\vartheta_-,\vartheta_+:(0,\infty)\to(0,\infty)$,
both tending to zero at zero, such that, uniformly in the points,
\[
 d_\times(z,z')<\vartheta_-(t)\ \Longrightarrow\ d_N(z,z')<t,
 \qquad
 d_N(z,z')<t\ \Longrightarrow\ d_\times(z,z')<\vartheta_+(t).
\]
For example, these moduli follow from uniform continuity of the identity
maps between the two compact metric realizations of $M\times Y$.
Choose $r=r(\eps)>0$ and $\delta_-(\eps)>0$, both tending to zero with
$\eps$, so that
\[
 \delta_-(\eps)<\vartheta_-(\eps),\qquad
 r(\eps)+\frac{\omega(\delta_-(\eps))}{1-\lambda}
 <\vartheta_-(\eps).
\]
Such choices are possible because $\omega(t)\to0$ as $t\to0$.
The upper comparison can be taken with
$\delta_+(\eps)=\vartheta_+(\eps)$ (or a slightly larger radius
if needed for strict inequalities). These choices are uniform in the
center and the block length.

Let $(x,y),(x',y')\in M\times Y$, and write their fiber coordinates after $j$
iterates as $y_j,y_j'$. Iterating the contraction/transverse-continuity bounds
gives
\[
d_Y(y_j,y_j')
\le \lambda^j d_Y(y,y')
+\sum_{k=0}^{j-1}\lambda^{j-1-k}
\omega\bigl(d_M(f^kx,f^kx')\bigr).
\]
Hence, if $x'\in B_f(x,\delta,n)$ and $d_Y(y,y')<r$, then
\[
d_Y(y_j,y_j')
\le r+\frac{\omega(\delta)}{1-\lambda},
\qquad 0\le j<n.
\]
By the preceding choices of $r(\eps)$ and $\delta_-(\eps)$, both the
base distance and the fiber distance at every iterate are smaller than
$\vartheta_-(\eps)$. The first metric implication therefore gives
\begin{equation}
\label{eq:skew-product-lower-inclusion}
\bigl\{(x',y'):\ x'\in B_f(x,\delta_-,n),\
 y'\in B_Y(y,r)\bigr\}
\subset B_F((x,y),\eps,n).
\end{equation}
By the second metric implication, with
$\delta_+(\eps)=\vartheta_+(\eps)\to0$,
\begin{equation}
\label{eq:skew-product-upper-inclusion}
B_F((x,y),\eps,n)
\subset \pi^{-1}(B_f(x,\delta_+,n)).
\end{equation}
Disintegrating and using
\eqref{eq:averaged-lower-fiber-ball-masses} with
$r=r(\eps)$ and $\delta=\delta_-(\eps)$ in the lower inclusion gives
\[
\widetilde\mu(B_F((x,y),\eps,n))
\ge e^{-q_n(r(\eps),\delta_-(\eps))}
\mu(B_f(x,\delta_-(\eps),n)),
\]
whereas the upper inclusion gives
\[
\widetilde\mu(B_F((x,y),\eps,n))
\le \mu(B_f(x,\delta_+(\eps),n)).
\]
Thus, forward comparability holds with
\[
a_n^+(\eps)=q_n(r(\eps),\delta_-(\eps))=o(n).
\]

For  backward comparison, fix a small base radius $\eta>0$. Choose
$\rho_-(\eta)>0$ so small that
$\delta_+(\rho_-(\eta))\le\eta$. Applying the upper inequality above at radius
$\rho_-(\eta)$ yields
\[
\widetilde\mu(B_F((x,y),\rho_-(\eta),n))
\le \mu(B_f(x,\eta,n)).
\]
Next, choose $\rho_+(\eta)>0$ so that
$\eta\le\delta_-(\rho_+(\eta))$ and $\rho_+(\eta)\to0$ as
$\eta\downarrow0$. Here is an explicit choice, which does not require
$\delta_-$ to be invertible. Fix $t_0>0$ sufficiently small and set
$t_k=2^{-k}t_0$. Since $\delta_-(t_k)>0$ for every $k$ and
$\delta_-(t_k)\to0$, for every sufficiently small $\eta>0$ the set
$\{k\ge0:\eta\le\delta_-(t_k)\}$ is nonempty and finite.
Let $k(\eta)$ be its largest element and put
$\rho_+(\eta)=t_{k(\eta)}$. For every fixed $K$, sufficiently
small $\eta$ satisfies $\eta\le\delta_-(t_K)$; hence
$k(\eta)\to\infty$ and $\rho_+(\eta)\to0$.
Applying the lower inequality at radius
$\rho_+(\eta)$ and using the monotonicity of base Bowen balls gives
\[
\mu(B_f(x,\eta,n))
\le e^{q_n(r(\rho_+(\eta)),\delta_-(\rho_+(\eta)))}
\widetilde\mu(B_F((x,y),\rho_+(\eta),n)).
\]
Therefore, backward comparability holds with
\[
a_n^-(\eta)
=q_n(r(\rho_+(\eta)),\delta_-(\rho_+(\eta)))=o(n).
\]
Hence, the two-sided subexponential Bowen-mass comparison follows.
\end{proof}

\begin{remark}
The pointwise lower bound
$\widetilde\mu_{x'}(B_Y(y,r))\ge c_r>0$ used in more uniform settings is a
special case of \eqref{eq:averaged-lower-fiber-ball-masses}, corresponding to
$q_n(r,\delta)\equiv |\log c_r|$. The present formulation allows spatial
nonuniformity, but the subexponential envelope must be uniform over the Gibbs
set. The mere pointwise convergence of center-dependent errors is not enough for
the uniform weak Gibbs definition used above.
\end{remark}

\begin{proposition}
\label{prop:local-product-bowen}
Let $G_F\subset N$ be measurable and suppose $\pi(G_F)$ is measurable.
Assume that for every $\hat x\in G_F$ there are
neighborhoods $U_{\hat x}^{\mathrm{base}}\subset M$,
$U_{\hat x}^{\mathrm{fib}}$ in a fixed local fiber model, a neighborhood
$R_{\hat x}\subset N$, and a homeomorphism
\[
\Theta_{\hat x}:
U_{\hat x}^{\mathrm{base}}\times U_{\hat x}^{\mathrm{fib}}
\longrightarrow R_{\hat x}
\]
such that
\[
\pi\circ\Theta_{\hat x}(z,\xi)=z
\]
whenever the expression is defined.

Assume that for every sufficiently small $\eps>0$ there are a base radius
$\delta_-(\eps)>0$ tending to zero with $\eps$, a fiber radius $r(\eps)>0$, measurable local product boxes
$Q_n(\hat x,\eps)\subset R_{\hat x}$, and a radius modulus
$\rho(\eps)\to0$ such that, for every $\hat x\in G_F$ and every $n\ge1$,
\emph{the full product set in the next display lies in the domain}
$U_{\hat x}^{\mathrm{base}}\times U_{\hat x}^{\mathrm{fib}}$
\emph{of the chart}, and
\begin{equation}
\label{eq:local-product-chart-lower}
\Theta_{\hat x}\!\left(
B_f(\pi\hat x,\delta_-(\eps),n)\times
B^{\mathrm{fib}}(\hat x,r(\eps))
\right)
\subset
Q_n(\hat x,\eps)
\subset
B_F(\hat x,\rho(\eps),n).
\end{equation}
The full-domain and uniform-in-time requirements in
\eqref{eq:local-product-chart-lower} are substantive assumptions:
the mere existence of local product charts does not imply that a fixed
base Bowen ball times a fixed fiber ball remains in a single chart for
all $n$ and all centers. In particular, the radii and the Bowen-ball
inclusions above must be verified uniformly on $G_F$ in each application;
no uniform chart size is inferred from pointwise local product structure.

Moreover, assume a lower product nondegeneracy estimate: for every sufficiently
small $\eps>0$ there are numbers $b_n(\eps)\ge0$ with $b_n(\eps)=o(n)$ and a
constant $c_\eps>0$ such that
\begin{align}
\label{eq:local-product-chart-mass}
&\widetilde\mu\!\left(
\Theta_{\hat x}\!\left(
B_f(\pi\hat x,\delta_-(\eps),n)\times
B^{\mathrm{fib}}(\hat x,r(\eps))
\right)\right)
\nonumber\\
&\hspace{2cm}\ge
c_\eps e^{-b_n(\eps)}
\mu(B_f(\pi\hat x,\delta_-(\eps),n))
\end{align}
for every $\hat x\in G_F$ and every $n\ge1$.

Then $(\widetilde\mu,\mu)$ satisfies both forward and backward
subexponential Bowen-mass comparability on $(G_F,\pi(G_F))$. More precisely,
for every small target radius, one may take logarithmic errors equal to
\[
\widehat b_n(\eps)+O_\eps(1),
\]
where $\widehat b_n(\eps)$ is the maximum of the finitely many
$b_n(\rho_j(\eps))$ generated by the radius reparametrizations used in the
forward and backward comparisons. Hence $\widehat b_n(\eps)=o(n)$. If the
lower product distortion is uniformly bounded in $n$, then
$a_n^\pm(\eps)=O_\eps(1)$.
\end{proposition}

\begin{proof}
The full chart-domain and Bowen-inclusion hypotheses ensure that the
product set and its image are defined and controlled for every center and
time. Neither property follows from local product coordinates alone.
The mass lower bound is an additional, independent assumption; in
particular, no lower conditional-mass estimate is deduced merely from
the existence of charts.
The first inclusion in \eqref{eq:local-product-chart-lower} and the lower mass
estimate \eqref{eq:local-product-chart-mass} give
\[
\widetilde\mu(B_F(\hat x,\rho(\eps),n))
\ge
c_\eps e^{-b_n(\eps)}
\mu(B_f(\pi\hat x,\delta_-(\eps),n)).
\]
On the other hand, uniform continuity of $\pi$ gives, for every sufficiently
small Bowen radius $\rho>0$, a base radius $\eta(\rho)\to0$ such that
\[
B_F(\hat x,\rho,n)
\subset
\pi^{-1}(B_f(\pi\hat x,\eta(\rho),n)),
\]
and therefore
\[
\widetilde\mu(B_F(\hat x,\rho,n))
\le
\mu(B_f(\pi\hat x,\eta(\rho),n)).
\]
We first make the forward comparison explicit. Given a small target Bowen
radius $\tau>0$, choose a parameter $\eps_\tau>0$ so that
$\rho(\eps_\tau)\le\tau$ and $\eps_\tau\le\tau$. Then monotonicity of Bowen balls gives
\[
\widetilde\mu(B_F(\hat x,\tau,n))
\ge c_{\eps_\tau}e^{-b_n(\eps_\tau)}
\mu(B_f(\pi\hat x,\delta_-(\eps_\tau),n)).
\]
Together with the projection estimate at radius $\tau$, this is forward
comparability with logarithmic error
$b_n(\eps_\tau)+|\log c_{\eps_\tau}|=o(n)$.

For backward comparability, fix a small base radius $\eta>0$. By the
monotone radius convention, choose a Bowen radius $\tau_-(\eta)>0$ such that
$\eta(\tau_-(\eta))\le\eta$. The projection estimate then gives
\[
\widetilde\mu(B_F(\hat x,\tau_-(\eta),n))
\le\mu(B_f(\pi\hat x,\eta,n)).
\]
Next, choose a parameter
$\eps_+(\eta)>0$ so that
\[
\eta\le\delta_-(\eps_+(\eta)),\qquad \eps_+(\eta)\to0\quad(\eta\downarrow0).
\]
Such a parameter exists, for example, by choosing a decreasing sequence
$t_k\downarrow0$ for which $\delta_-(t_k)\downarrow0$ after passing
to a subsequence and taking a $t_k$ with the largest possible index satisfying
$\delta_-(t_k)\ge\eta$. This choice tends to zero as $\eta\downarrow0$;
no inverse function or strict monotonicity is needed.
The lower product estimate and monotonicity of base Bowen balls yield
\[
\mu(B_f(\pi\hat x,\eta,n))
\le c_{\eps_+(\eta)}^{-1}e^{b_n(\eps_+(\eta))}
\widetilde\mu(B_F(\hat x,\rho(\eps_+(\eta)),n)).
\]
Thus the backward logarithmic error is
$b_n(\eps_+(\eta))+|\log c_{\eps_+(\eta)}|=o(n)$. Only finitely many
radius reparametrizations are used for each target radius, so one may collect
their errors into
\[
\widehat b_n(\eps)+O_\eps(1),
\qquad
\widehat b_n(\eps)=\max_j b_n(\rho_j(\eps))=o(n).
\]
This gives the two-sided comparison.
\end{proof}

\section{Sequential Gibbs Measures and Fiber-Contraction}\label{sec:sequential-gibbs}

This section proves item~\ref{item:main-sequential} of
Theorem~\ref{thm:master-reduction}. The weak Gibbs results above require
estimates at every time. In many
non-uniformly hyperbolic situations, however, Gibbs estimates are available
only along a point-dependent subsequence of times. This motivates a
sequential version of the transfer mechanism. We formulate the estimates for Bowen balls. Their relation to the
symbolic cylinder definition of Ferreira--Oliveira \cite{FerreiraOliveira}
is discussed at the end of the section.

\begin{definition}[Bowen sequential Gibbs]
\label{def:bowen-sequential-gibbs}
Let $T:X\to X$ be continuous, let $\phi\in C(X)$, let $\nu$ be a probability
measure, and let $G\subset X$ be measurable with $\nu(G)=1$. We say that
$\nu$ is \emph{Bowen sequential Gibbs for $(T,\phi)$ on $G$} if there exist
$P\in\mathbb R$ and, for every $x\in G$, an increasing sequence
$n_1(x)<n_2(x)<\cdots$, chosen independently of the Bowen radius, such that
the following uniform property holds: for every sufficiently small
$\eps>0$ there exists $K(\eps)\ge1$, independent of $x$, $i$, and $j$,
such that for every $x\in G$, every $i\ge1$, and every
$0\le j<n_i(x)$,
\begin{equation}
 K(\eps)^{-1}
 \le
 \frac{\nu\bigl(B_T(T^jx,\eps,n_i(x)-j)\bigr)}
 {\exp\bigl(S_{n_i(x)-j}\phi(T^jx)-(n_i(x)-j)P\bigr)}
 \le K(\eps).
 \label{eq:bowen-sequential-gibbs}
\end{equation}
Whenever Gibbs-time sequences are used as measurable functions, we require
each map $x\mapsto n_i(x)$ to be measurable.
Any such sequence is called an \emph{admissible Bowen Gibbs-time sequence}
for $x$, and $P$ is called the \emph{sequential Gibbs pressure} associated
with these estimates.
\end{definition}

\begin{remark}
The constant $K(\eps)$ in Definition~\ref{def:bowen-sequential-gibbs} is
uniform on the full-measure Gibbs set $G$. Only the admissible sequence
$n_i(x)$ is allowed to depend on the center. This matches the uniform Gibbs
constant in the classical symbolic sequential Gibbs formulation, while the
Bowen radii are permitted to vary through the radius reparametrizations used
below.
\end{remark}

The tail requirement in \eqref{eq:bowen-sequential-gibbs} mirrors the
standard symbolic definition of sequential Gibbs measures, in which the
same estimate is required for every suffix of a Gibbs cylinder. We do not
identify $P$ a priori with the full topological pressure: in the symbolic theory,
it is naturally associated with the pressure carried by the Gibbs set.

\begin{definition}[Blockwise Bowen sequential weak Gibbs]
\label{def:bowen-sequential-weak}
With the notation above, $\nu$ is \emph{blockwise Bowen sequential weak
Gibbs} with sequential pressure $P$ if, for every $x\in G$, there is an
increasing sequence $n_i(x)\to\infty$, chosen independently of the Bowen
radius, such that for every sufficiently small $\eps>0$ there are constants
$K_i(x,\eps)\ge1$ for which \eqref{eq:bowen-sequential-gibbs} holds with
$K(\eps)$ replaced by $K_i(x,\eps)$, uniformly for all
$0\le j<n_i(x)$, and
\begin{equation}
 \frac{1}{n_i(x)}\log K_i(x,\eps)\longrightarrow0.
 \label{eq:sequential-weak-subexp}
\end{equation}
\end{definition}

\begin{remark}
The error in Definition~\ref{def:bowen-sequential-weak} is normalized by the
full Gibbs-block length $n_i(x)$, not by the suffix length $n_i(x)-j$. Thus, the
definition is intentionally weaker on very short suffixes. The adjective
\emph{blockwise} will be used to avoid confusing this convention with stronger
sequential weak Gibbs notions, in which the error is required to be
subexponential relative to every suffix length.
\end{remark}

\begin{definition}[Forward sequential Bowen-mass comparability]
\label{def:sequential-forward-comparison}
Let $\widetilde\mu\leftrightarrow\mu$ be the corresponding invariant measures,
let $G_f\subset M$ have the full $\mu$-measure, and suppose that an admissible
sequence $n_i(x)$ is fixed for every $x\in G_f$. We say that \emph{forward
sequential Bowen-mass comparability} holds if there is a full
$\widetilde\mu$-measure set $G_F\subset\pi^{-1}(G_f)$ such that, for every
sufficiently small $\eps>0$, there are radii
$\eps_-(\eps),\eps_+(\eps)>0$ tending to zero with $\eps$ and numbers
$a_i(\hat x,\eps)\ge0$ satisfying
\begin{equation}
 \frac{a_i(\hat x,\eps)}{n_i(\pi\hat x)}\longrightarrow0
 \qquad\text{for every }\hat x\in G_F,
 \label{eq:seq-mass-error-subexp}
\end{equation}
such that, writing $x=\pi\hat x$ and $m=n_i(x)-j$, for every
$0\le j<n_i(x)$,
\begin{align}
 e^{-a_i(\hat x,\eps)}
 \mu\bigl(B_f(f^jx,\eps_-,m)\bigr)
 &\le
 \widetilde\mu\bigl(B_F(F^j\hat x,\eps,m)\bigr)
 \nonumber\\
 &\le
 e^{a_i(\hat x,\eps)}
 \mu\bigl(B_f(f^jx,\eps_+,m)\bigr).
 \label{eq:seq-forward-comparison}
\end{align}
If, for each fixed $\eps$, the errors $a_i(\hat x,\eps)$ are 
uniformly bounded in $\hat x$ and $i$, we say that the comparison is
\emph{uniformly bounded}.
\end{definition}

\begin{definition}[Backward sequential Bowen-mass comparability]
\label{def:sequential-backward-comparison}
Let $G_f\subset M$ be measurable with full $\mu$-measure. Assume
that the Gibbs-time maps $\hat x\mapsto n_i(\hat x)$ are measurable
on $G_F$, and that a measurable
prescribed lift
\[
 \ell:G_f\to N,
 \qquad \pi\circ\ell=\operatorname{id}_{G_f},
\]
is fixed, with $\ell(G_f)$ contained in a full $\widetilde\mu$-measure set
$G_F$ on which an admissible Gibbs-time sequence $n_i(\hat x)$ is defined.
Set
\[
 n_i(x):=n_i(\ell(x)).
\]
These maps are measurable by composition. The full-measure condition on
$G_F$ does not by itself imply that $\ell(G_f)$ lies in $G_F$;
this inclusion is an independent assumption.
We say that \emph{backward sequential Bowen-mass comparability along $\ell$}
holds if, for every sufficiently small $\eps>0$, there are radii
$\eps_-(\eps),\eps_+(\eps)>0$ tending to zero with $\eps$ and numbers
$b_i(x,\eps)\ge0$ satisfying
\[
 \frac{b_i(x,\eps)}{n_i(x)}\longrightarrow0
\]
such that, for every $x\in G_f$, every $0\le j<n_i(x)$, and
$m=n_i(x)-j$,
\begin{align}
 e^{-b_i(x,\eps)}
 \widetilde\mu\bigl(B_F(F^j\ell(x),\eps_-,m)\bigr)
 &\le \mu\bigl(B_f(f^jx,\eps,m)\bigr)
 \nonumber\\
 &\le e^{b_i(x,\eps)}
 \widetilde\mu\bigl(B_F(F^j\ell(x),\eps_+,m)\bigr).
 \label{eq:seq-backward-comparison}
\end{align}
If  $b_i(x,\eps)$ are uniformly bounded in $x$ and $i$ for fixed
$\eps$, we call the comparison \emph{uniformly bounded}.
\end{definition}

\begin{theorem}[Forward blockwise sequential weak Gibbs transfer]
\label{thm:sequential-weak-forward}
Assume
\[
 \varphi=\psi\circ\pi+u-u\circ F,
\]
and let $\widetilde\mu\leftrightarrow\mu$. Suppose $\mu$ is blockwise Bowen
sequential weak Gibbs for $(f,\psi)$ with sequential pressure $P$ and a fixed
admissible sequence $n_i(x)$ on a full set $G_f$. If forward sequential
Bowen-mass comparability holds, then $\widetilde\mu$ is blockwise Bowen
sequential weak Gibbs for $(F,\varphi)$ with the same sequential pressure $P$.
For $\widetilde\mu$-a.e. $\hat x$, the prescribed sequence
$n_i(\pi\hat x)$ remains admissible upstairs.
\end{theorem}

\begin{proof}
Fix $\hat x\in G_F$, put $x=\pi\hat x$, choose a Gibbs time
$n_i=n_i(x)$, and let $0\le j<n_i$. Put $m=n_i-j$. The cohomological identity
gives
\begin{equation}
 S_m\varphi(F^j\hat x)
 =S_m\psi(f^jx)+u(F^j\hat x)-u(F^{n_i}\hat x),
 \label{eq:seq-birkhoff-coboundary}
\end{equation}
so
\begin{equation}
 \left|S_m\varphi(F^j\hat x)-S_m\psi(f^jx)\right|
 \le2\|u\|_\infty.
 \label{eq:seq-birkhoff-error}
\end{equation}
By Definition~\ref{def:bowen-sequential-weak}, the same sequence $n_i(x)$
works simultaneously in the reparametrized radii $\eps_-$ and $\eps_+$.
Combining those estimates with \eqref{eq:seq-forward-comparison} gives a
Gibbs constant upstairs bounded by
\[
 \widetilde K_i(\hat x,\eps)
 \le
 \exp\bigl(a_i(\hat x,\eps)+2\|u\|_\infty\bigr)
 \max\{K_i(x,\eps_-),K_i(x,\eps_+)\}.
\]
The pressure $P$ is unchanged: the
cohomological correction contributes only to the bounded endpoint term in
\eqref{eq:seq-birkhoff-coboundary}. Hence
\[
 \frac1{n_i(x)}\log \widetilde K_i(\hat x,\eps)\longrightarrow0.
\]
The bound holds for all suffixes $0\le j<n_i(x)$, and the prescribed sequence remains admissible in the extension.
\end{proof}

\begin{theorem}[Backward blockwise sequential weak Gibbs transfer]
\label{thm:sequential-weak-backward}
Assume
\[
 \varphi=\psi\circ\pi+u-u\circ F,
\]
and let $\widetilde\mu\leftrightarrow\mu$. Suppose $\widetilde\mu$ is
blockwise Bowen sequential weak Gibbs for $(F,\varphi)$ with sequential
pressure $P$ on a full set $G_F$, and let $\ell:G_f\to G_F$ be a prescribed
lift as in Definition~\ref{def:sequential-backward-comparison}. If backward
sequential Bowen-mass comparability along $\ell$ holds, then $\mu$ is
blockwise Bowen sequential weak Gibbs for $(f,\psi)$ with the same sequential
pressure $P$, along the admissible sequence
\[
 n_i(x)=n_i(\ell(x)).
\]
\end{theorem}

\begin{proof}
Fix $x\in G_f$, write $\hat x=\ell(x)$, and let
$n_i=n_i(\hat x)$. For $0\le j<n_i$ put $m=n_i-j$. The identity
\eqref{eq:seq-birkhoff-coboundary}, applied to $\hat x$, gives the same bounded
cohomological error \eqref{eq:seq-birkhoff-error}. Combining the sequential
weak Gibbs estimates for $\widetilde\mu$ at the radii $\eps_-$ and $\eps_+$
with \eqref{eq:seq-backward-comparison} yields a quotient Gibbs constant
bounded by
\[
 K_i^{\mathrm{base}}(x,\eps)
 \le
 \exp\bigl(b_i(x,\eps)+2\|u\|_\infty\bigr)
 \max\{\widetilde K_i(\hat x,\eps_-),\widetilde K_i(\hat x,\eps_+)\}.
\]
Dividing the logarithm by $n_i(x)$ and using the two subexponential
hypotheses establishes the blockwise sequential weak Gibbs estimate.
\end{proof}

\begin{corollary}
\label{cor:sequential-gibbs-transfer}
Under the hypotheses of Theorem~\ref{thm:sequential-weak-forward}, assume
that the quotient measure is Bowen sequential Gibbs in the sense of
Definition~\ref{def:bowen-sequential-gibbs} and that forward sequential
Bowen-mass comparability is uniformly bounded. Then the lifted measure is
Bowen sequential Gibbs with the same sequential pressure, and the prescribed
admissible Gibbs-time sequence used in the forward comparison remains
admissible upstairs.

Conversely, under the hypotheses of Theorem~\ref{thm:sequential-weak-backward},
if the extension measure is Bowen sequential Gibbs and the backward
comparison is uniformly bounded, then the quotient measure is Bowen
sequential Gibbs along the sequence induced by the prescribed lift.
\end{corollary}

\begin{proof}
In the forward direction, for each fixed radius the proof of
Theorem~\ref{thm:sequential-weak-forward} gives a Gibbs constant bounded by
\[
 e^{A(\eps)+2\|u\|_\infty}
 \max\{K(\eps_-),K(\eps_+)\},
\]
where $A(\eps)$ is a uniform bound for the logarithmic mass error. This
constant is independent of the Gibbs index and the suffix. The backward
statement follows identically from
Theorem~\ref{thm:sequential-weak-backward}.
\end{proof}

\begin{remark}
An $o(n_i)$ logarithmic mass error is sufficient for the blockwise sequential
weak Gibbs transfer, but it does not in general preserve a fixed Gibbs
constant. Uniformly bounded mass distortion is therefore the natural
hypothesis for preserving the strong Bowen sequential Gibbs property.
\end{remark}

\begin{remark}
\label{rem:factor-sequential-gibbs}
Ferreira and Oliveira \cite{FerreiraOliveira} prove, for regular one-block
factor maps between one-sided symbolic systems, that the image of a sequential
Gibbs measure is sequential Gibbs and that the maximal Gibbs-time sequence is
preserved in their cylinder framework. Their result extends the classical
factor-Gibbs problem studied, among others, in
\cite{KemptonPollicott,ChazottesUgalde,KemptonSFT,Piraino}. Their result is a projection theorem
and constructs a potential on the symbolic factor. The mechanism above is
complementary: for a fiber-contracting extension it gives a lifting theorem
once the potential has been canonically reduced to the quotient and the
required Bowen-mass comparison is available.

Our Bowen formulation should not be identified with the cylinder formulation
without an additional symbolic comparison lemma. In particular, for two-sided
shifts a small Bowen ball is a bilateral cylinder whose time window is enlarged
by a radius-dependent but fixed number of coordinates. Thus, the relation with
Ferreira--Oliveira is exact only in symbolic settings where this finite-window
change has uniformly bounded measure distortion along the relevant Gibbs
blocks.
\end{remark}

\begin{lemma}
\label{lem:cylinder-bowen-bridge}
Let $(\Sigma,\sigma)$ be a symbolic system endowed with a standard symbolic
metric, and let $\mu$ be a probability measure. Suppose that for every
sufficiently small Bowen radius $\eps>0$ there is a constant
$C(\eps)\ge1$ such that, for every relevant orbit block, every $n\ge1$, and
every $0\le j<n$,
\begin{equation}
 C(\eps)^{-1}
 \mu([x_j\ldots x_{n-1}])
 \le
 \mu\bigl(B_\sigma(\sigma^jx,\eps,n-j)\bigr)
 \le
 C(\eps)
 \mu([x_j\ldots x_{n-1}]).
 \label{eq:cylinder-bowen-suffix-comparison}
\end{equation}
Here, $[x_j\ldots x_{n-1}]$ denotes the suffix cylinder used in the symbolic
sequential Gibbs definition. In a two-sided shift, condition
\eqref{eq:cylinder-bowen-suffix-comparison} is implied, for instance, when a
Bowen ball is the corresponding radius-dependent bilateral enlargement of
that suffix cylinder and the measure of this enlarged cylinder is uniformly
comparable, by a factor depending only on $\eps$, with the measure of the
central suffix cylinder.

If the cylinder estimates defining sequential Gibbs hold along a prescribed
admissible sequence $n_i(x)$, then the same sequence is admissible for the
Bowen sequential Gibbs definition, with Gibbs constants multiplied by at most
$C(\eps)$.
\end{lemma}

\begin{proof}
Fix a Gibbs block $n_i(x)$ and a suffix $0\le j<n_i(x)$. Apply the symbolic
cylinder Gibbs estimate to $[x_j\ldots x_{n_i(x)-1}]$ and then use
\eqref{eq:cylinder-bowen-suffix-comparison} with $n=n_i(x)$. The resulting
upper and lower Bowen-ball estimates acquire at most the multiplicative factor
$C(\eps)$. Since this factor is independent of the Gibbs index $i$ and of the
suffix $j$, the same prescribed sequence remains admissible in the Bowen
formulation.
\end{proof}

\section{Subcohomology}\label{sec:subcohomology}

\subsection{The quotient property}

\begin{definition}[Subcohomological property]
A dynamical system $(M,f)$ has property $\mathrm{SC}$ for a class
$\mathcal C$ of potentials if, for every $\psi\in\mathcal C$,
\[
m(\psi,f)\geq0
\quad\Longrightarrow\quad
\exists v\in C(M):
\qquad
\psi\geq v-v\circ f.
\]
\end{definition}

\begin{theorem}[Subcohomology transfer]
\label{thm:SC-transfer}
Assume the cohomological reduction
\[
\varphi=\psi\circ\pi+u-u\circ F
\]
and suppose $(M,f)$ has property $\mathrm{SC}$ for the class containing
$\psi$. If
\[
m(\varphi,F)\geq0,
\]
then there exists $U\in C(N)$ such that
\[
\boxed{\varphi\geq U-U\circ F.}
\]
If $u$ and $v\circ\pi$ are H\"older, where
\[
\psi\geq v-v\circ f,
\]
then
\[
U=u+v\circ\pi
\]
is H\"older. In particular, this conclusion holds if $v$ is
$\alpha$-H\"older and $\pi$ is $\tau$-H\"older, in which case
$v\circ\pi$ is $\alpha\tau$-H\"older.
\end{theorem}

\begin{proof}
By Corollary~\ref{cor:average-correspondence},
\[
m(\psi,f)=m(\varphi,F)\geq0.
\]
The property $\mathrm{SC}$ gives $v\in C(M)$ satisfying
\[
\psi\geq v-v\circ f.
\]
Set
\[
U=u+v\circ\pi.
\]
Using $\pi\circ F=f\circ\pi$,
\begin{align}
U-U\circ F
&=
u-u\circ F
+
(v-v\circ f)\circ\pi\\
&\leq
u-u\circ F+\psi\circ\pi\\
&=\varphi.
\end{align}
The claim follows.
\end{proof}

\begin{corollary}
\label{cor:zooming}
Suppose that the quotient has property $\mathrm{SC}$ for a class containing the
reduced potential $\psi$. Then $m(\varphi,F)\ge0$ implies a continuous
subaction upstairs. A H\"older conclusion additionally requires H\"older
regularity of $u$ and $v\circ\pi$. No critical subaction property is
inferred here merely from the label ``zooming''.
\end{corollary}

\begin{remark}
The weighted construction
\[
v=\sum_{j=0}^{\infty}\theta^j|\phi\circ T^j|
\]
gives
\[
v\circ T-v
=
\left(\frac1\theta-1\right)
\sum_{j=1}^{\infty}\theta^j|\phi\circ T^j|
-|\phi|.
\]
The uniform estimate for the first term is
\[
\left(\frac1\theta-1\right)
\frac{\theta}{1-\theta}\|\phi\|_\infty
=
\|\phi\|_\infty.
\]
Thus, this term does \emph{not} tend to zero as $\theta\to1^-$. Hence,
this weighted-series estimate cannot establish the critical case. The
regime $m(\phi,T)=0$ requires a genuine subcohomological theorem.
\end{remark}

\subsection{Positive margin and a Mañ\'e-type construction}

For a continuous map $T:X\to X$ of a compact metric space and
$\phi\in C(X)$, write
\[
 m(\phi,T)
 :=\min_{\mu\in\mathcal M_T(X)}\int_X\phi\,d\mu.
\]
 The positive-margin case admits a direct argument
which is conceptually different from the critical case.

\begin{lemma}[Uniform growth of Birkhoff sums]
\label{lem:uniform-growth}
Let $T:X\to X$ be continuous on compact $X$ and $\phi\in C(X)$. If
\[
m(\phi,T)>0,
\]
then there exist $c>0$ and $N_0\in\mathbb N$ such that
\[
S_n\phi(x)
=
\sum_{j=0}^{n-1}\phi(T^jx)
\geq cn
\]
for every $x\in X$ and every $n\geq N_0$.
\end{lemma}

\begin{proof}
Choose $0<c<m(\phi,T)$. Suppose that the conclusion is false. Then there are
$n_k\to\infty$ and $x_k\in X$ such that
\[
S_{n_k}\phi(x_k)<cn_k.
\]
Define the empirical measures
\[
\mu_k
=
\frac1{n_k}
\sum_{j=0}^{n_k-1}\delta_{T^jx_k}.
\]
By compactness of the space of probability measures, after passing to a
subsequence $\mu_k\to\mu$ in the weak-$*$ topology. Standard telescoping
gives, for every $\psi\in C(X)$,
\[
\left|
\int\psi\circ T\,d\mu_k-\int\psi\,d\mu_k
\right|
\leq
\frac{2\|\psi\|_\infty}{n_k}\to0.
\]
Hence $\mu$ is $T$-invariant. On the other hand,
\[
\int\phi\,d\mu
=
\lim_{k\to\infty}
\frac{S_{n_k}\phi(x_k)}{n_k}
\leq c<m(\phi,T),
\]
a contradiction.
\end{proof}

\begin{proposition}
\label{prop:positive}
If $m(\phi,T)>0$, then there exists $u\in C(X)$ such that
\[
\phi\geq u-u\circ T.
\]
\end{proposition}

\begin{proof}
By Lemma~\ref{lem:uniform-growth}, choose $c>0$ and $N_0$ such that
$S_n\phi(x)\ge cn$ for every $n\ge N_0$ and every $x$. Define
\begin{equation}
 u(x)=\inf_{n\ge0}S_n\phi(x),\qquad S_0\phi=0.
 \label{eq:positive-inf}
\end{equation}
For $n\ge N_0$ one has $S_n\phi(x)\ge cn>0=S_0\phi(x)$. Hence
\[
 u(x)=\min_{0\le n<N_0}S_n\phi(x),
\]
a minimum of finitely many continuous functions; therefore $u$ is
continuous.

Moreover,
\begin{align*}
\phi(x)+u(Tx)
&=\phi(x)+\inf_{n\ge0}S_n\phi(Tx)\\
&=\inf_{n\ge0}\bigl(\phi(x)+S_n\phi(Tx)\bigr)\\
&=\inf_{m\ge1}S_m\phi(x)\\
&\ge \inf_{m\ge0}S_m\phi(x)=u(x).
\end{align*}
Thus, $\phi(x)\ge u(x)-u(Tx)$ for every $x$.
\end{proof}

\begin{remark}
A positive margin also gives a Lipschitz subaction without a quotient
regularity theorem. Choose $0<c<m(\phi,T)$ and apply the preceding
construction to $\phi-c$. It gives a continuous $w$ with
$\phi\ge w-w\circ T+c$. Lipschitz functions are uniformly dense in
$C(X)$ for compact metric $X$; choose a Lipschitz $v$ with
$\|v-w\|_\infty<c/4$. Then
\[
 \phi\ge v-v\circ T+c-2\|v-w\|_\infty
 \ge v-v\circ T+c/2.
\]

\end{remark}

\section{Liv\v{s}ic Transfer}\label{sec:livsic}

Uniform fiber contraction gives the periodic lifting used below.

\begin{lemma}
\label{lem:periodic-lifting}
If $p\in M$ satisfies $f^np=p$, then there exists a unique
$\hat p\in N_p$ such that $F^n\hat p=\hat p$.
\end{lemma}

\begin{proof}
Because $\pi$ is continuous and $N$ is compact, the fiber
$N_p=\pi^{-1}(p)$ is compact and therefore complete. Since
$f^np=p$, semiconjugacy gives $F^n(N_p)\subset N_p$. For
$\hat x,\hat y\in N_p$,
\[
 d_N(F^n\hat x,F^n\hat y)\le\lambda^n d_N(\hat x,\hat y).
\]
Thus $F^n|_{N_p}$ is a strict contraction of the complete metric space
$N_p$ into itself. The Banach fixed point theorem gives a unique fixed
point $\hat p\in N_p$.
\end{proof}

\begin{theorem}[Liv\v{s}ic transfer]
\label{thm:livsic}
Assume a continuous cohomological identity as in
Proposition~\ref{thm:reduction} (or verified directly), and suppose that
the quotient admits a
Liv\v{s}ic theorem for the regularity class of the induced potential
$\psi$. Let $\varphi$ satisfy
\[
\sum_{j=0}^{n-1}\varphi(F^j\hat p)=0
\]
whenever $F^n\hat p=\hat p$. Then there exists a quotient transfer $v$
with $\psi=v-v\circ f$, and
\[
 \varphi=U-U\circ F,\qquad U=u+v\circ\pi.
\]
In particular, if $u$ and $v\circ\pi$ are H\"older, then $U$ is
H\"older.
\end{theorem}

\begin{proof}
Write $\varphi=\psi\circ\pi+u-u\circ F$. If $f^np=p$, let
$\hat p\in N_p$ be the fixed point of $F^n$ given by
Lemma~\ref{lem:periodic-lifting}. Then
\begin{align*}
0
&=\sum_{j=0}^{n-1}\varphi(F^j\hat p)\\
&=\sum_{j=0}^{n-1}\psi(f^jp)+u(\hat p)-u(F^n\hat p)\\
&=\sum_{j=0}^{n-1}\psi(f^jp).
\end{align*}
Hence, the periodic obstruction vanishes for $\psi$. The quotient
Liv\v{s}ic theorem yields $\psi=v-v\circ f$, and semiconjugacy gives
\[
\varphi=(v-v\circ f)\circ\pi+u-u\circ F
=(u+v\circ\pi)-(u+v\circ\pi)\circ F.
\]
\end{proof}

\begin{corollary}[Liv\v{s}ic from two-sided subcohomology]
\label{cor:two-sided-sc-livsic}
Assume the compact semiconjugacy and uniform fiber contraction hypotheses
of Section~\ref{sec:geom}. Let $\varphi,u\in C(N)$ and $\psi\in C(M)$
satisfy the continuous cohomological identity
\[
\varphi=\psi\circ\pi+u-u\circ F.
\]
Assume that periodic points of $f$ are dense and that the quotient has
property $\mathrm{SC}$ for both $\psi$ and $-\psi$. Suppose
\[
\int_N\varphi\,d\widetilde\mu=0
\qquad\text{for every }\widetilde\mu\in\M_F(N).
\]
Then there exists $U\in C(N)$ such that $\varphi=U-U\circ F$.
If the quotient subactions and the reduction transfer have H\"older
regularity, then $U$ is H\"older.
\end{corollary}

\begin{proof}
By Corollary~\ref{cor:average-correspondence}, every
$\mu\in\M_f(M)$ satisfies $\int\psi\,d\mu=0$. The property $\mathrm{SC}$ applied to $\psi$ gives a continuous $v_1$ with
\[
v_1-v_1\circ f\le\psi.
\]
Applying $\mathrm{SC}$ to $-\psi$ gives a continuous function $\widetilde v_2$
with
\[
\widetilde v_2-\widetilde v_2\circ f\le-\psi.
\]
Set $v_2=-\widetilde v_2$. Then
\[
\psi\le v_2-v_2\circ f,
\]
and therefore
\[
v_1-v_1\circ f\le\psi\le v_2-v_2\circ f.
\]
Set $w=v_2-v_1$. Then $q:=w-w\circ f\ge0$. If $p$ is periodic for
period $n$, telescoping gives
\[
0=w(p)-w(f^np)=\sum_{j=0}^{n-1}q(f^jp).
\]
Every summand is nonnegative; hence $q=0$ in every periodic orbit. The density
of periodic points and continuity imply $q\equiv0$. Thus, the lower and
upper coboundaries coincide, so
$\psi=v_1-v_1\circ f$. The assumed cohomological identity then yields the required
coboundary on $N$, with $U=u+v_1\circ\pi$.
\end{proof}

\subsection{Proof of the general reduction and transfer principle}

\begin{proof}[Proof of Theorem~\ref{thm:master-reduction}]
The continuous section is used only for the construction in
item~\ref{item:main-normalization}; the remaining conclusions use their
explicit additional hypotheses. Item~\ref{item:main-normalization} is
Proposition~\ref{thm:reduction}, with the regularity supplement of
Proposition~\ref{prop:holder-reduction}.
Theorem~\ref{thm:lift-bijection} proves the invariant-measure bijection
without a continuous section;
the affine and ergodic assertions are justified in the proof of
Theorem~\ref{thm:ergodic-optimization}.
Lemma~\ref{lem:relative-zero} and Theorem~\ref{thm:pressure} give the
entropy assertions in item~\ref{item:main-measures}.
Corollary~\ref{cor:average-correspondence} and
Theorems~\ref{thm:ergodic-optimization} and~\ref{thm:pressure} prove
item~\ref{item:main-thermodynamics} for any verified continuous identity.

Theorem~\ref{thm:SC-transfer} proves the conditional subaction assertion
in item~\ref{item:main-subcohomology}. Proposition~\ref{prop:positive}
and its following approximation argument give the independent Lipschitz
subaction assertion recorded after the theorem under a strictly positive margin.
Lemma~\ref{lem:periodic-lifting}, Theorem~\ref{thm:livsic}, and
Corollary~\ref{cor:two-sided-sc-livsic} prove
item~\ref{item:main-livsic} with their respective quotient hypotheses.
The formula $U=u+v\circ\pi$ gives the asserted regularity whenever both
summands are H\"older.

Theorems~\ref{thm:weak-gibbs-forward} and~\ref{thm:weak-gibbs-backward},
together with Corollary~\ref{thm:weak-gibbs}, give
item~\ref{item:main-gibbs}, including the explicit logarithmic error.
Finally, Theorems~\ref{thm:sequential-weak-forward}
and~\ref{thm:sequential-weak-backward} and
Corollary~\ref{cor:sequential-gibbs-transfer} prove
item~\ref{item:main-sequential}. Their proofs use only the continuous
potential identity and the prescribed sequential comparisons. This also
justifies the continuous non-H\"older scope of the final assertion.
\end{proof}

\section{Explicit Examples and Applications}
\label{sec:examples}

We begin with an affine model for which the potential identity can be
computed in closed form. Symbolic examples then compare bounded and
logarithmic mass errors. The renewal model is the main application, and
the bilateral counterexample explains why uniformity in the centers matters.

\subsection{Affine contracting skew-products: an explicit reduction}

Let $M=\mathbb S^1=\mathbb R/\mathbb Z$ and let
\[
f(x)=dx\pmod 1,\qquad d\ge2.
\]
Fix $R>0$, $|a|<1$, and a Lipschitz function
$\rho:\mathbb S^1\to\mathbb R$ satisfying
\[
\|\rho\|_\infty\le (1-|a|)R.
\]
Then
\[
F:\mathbb S^1\times[-R,R]\longrightarrow
\mathbb S^1\times[-R,R],
\qquad
F(x,y)=\bigl(f(x),ay+\rho(x)\bigr)
\]
is a well-defined continuous skew-product. We take
\[
\pi(x,y)=x.
\]
The fibers are vertical intervals and
\[
|y_n-y_n'|=|a|^n|y-y'|
\]
whenever the initial points have the same base coordinate. Thus the selector hypothesis of Section~\ref{ass:selector}
holds with the continuous section
\[
s(x)=(x,0),
\qquad
r(x,y)=(x,0).
\]

The next proposition computes the reduction for a class of potentials that
is simple enough to be explicit but is not constant along the contracting
fibers.

\begin{proposition}[Explicit affine reduction]
\label{prop:affine-explicit}
Let
\[
F(x,y)=\bigl(f(x),ay+\rho(x)\bigr),\qquad |a|<1,
\]
as above, and let
\[
\varphi(x,y)=g(x)+by,
\]
where $g:\mathbb S^1\to\mathbb R$ is continuous and $b\in\mathbb R$.
Then the stable-difference series associated with the selector $s(x)=(x,0)$
converges explicitly and yields
\[
\boxed{u(x,y)=\frac{b}{1-a}\,y}.
\]
If $g$ is H\"older, this is precisely the transfer function furnished by
Proposition~\ref{thm:reduction}, and the reduced potential is
\[
\boxed{\psi(x)=g(x)+\frac{b}{1-a}\rho(x).}
\]
Consequently
\[
\boxed{
g(x)+by
=
g(x)+\frac{b}{1-a}\rho(x)
+\frac{b}{1-a}y
-\frac{b}{1-a}\bigl(ay+\rho(x)\bigr).
}
\]
\end{proposition}

\begin{proof}
For points in the same fiber,
\[
F^j(x,y)=\bigl(f^j(x),y_j\bigr),
\]
where
\[
y_j=a^jy+
\sum_{k=0}^{j-1}a^{j-1-k}\rho(f^kx).
\]
For the selected point $(x,0)$ the corresponding fiber coordinate is
\[
y_j^0=
\sum_{k=0}^{j-1}a^{j-1-k}\rho(f^kx).
\]
Hence
\[
\varphi(F^j(x,y))-\varphi(F^j(x,0))
=b(y_j-y_j^0)=ba^jy.
\]
The stable-difference series therefore gives
\[
u(x,y)=\sum_{j=0}^\infty ba^jy
=\frac{b}{1-a}y.
\]
Finally,
\[
\begin{aligned}
\varphi-u+u\circ F
&=g(x)+by-\frac{b}{1-a}y
+\frac{b}{1-a}(ay+\rho(x))\\
&=g(x)+\frac{b}{1-a}\rho(x),
\end{aligned}
\]
which depends only on $x$.
\end{proof}

\begin{remark}
For $b\ne0$, the potential $\varphi(x,y)=g(x)+by$ varies along each
fiber. The computation writes this dependence as a coboundary and the
quotient correction
\[
\frac{b}{1-a}\rho(x)
\]
on the base. Thus the formula records how the contracting coordinate
changes the reduced potential.
\end{remark}

\subsubsection{Consequences for ergodic optimization}

For the affine model,
\[
m(\varphi,F)
=
\min_{\mu\in\M_f(\mathbb S^1)}
\int_{\mathbb S^1}
\left(g+\frac{b}{1-a}\rho\right)\,d\mu.
\]
Thus, a minimization problem for the genuinely two-variable observable
$g(x)+by$ is exactly the one-dimensional optimization problem for
\[
\psi=g+\frac{b}{1-a}\rho.
\]
Moreover, minimizing measures correspond bijectively. In particular, if the
quotient potential has a unique minimizing measure $\mu_{\min}$, then the
extension has a unique minimizing measure, namely its unique invariant lift.

For $0\le a<1$, the factor $1/(1-a)$ has the expected memory
interpretation: stronger contraction (smaller $a$) reduces the cumulative
contribution of the fiber coordinate, whereas as $a\to1^-$ that accumulated
contribution becomes large. For negative $a$ the alternating signs change
this monotone interpretation, although the explicit formula remains valid
for every $|a|<1$.

\subsubsection{Consequences for pressure and equilibrium states}

Suppose $g$ and $\rho$ are H\"older. Then
\[
\psi(x)=g(x)+\frac{b}{1-a}\rho(x)
\]
is H\"older. Since $f(x)=dx\pmod1$ is an expanding circle map, the standard
thermodynamic formalism for H\"older potentials gives an equilibrium state
for $\psi$. Theorem~\ref{thm:pressure} then yields
\[
P_F(\varphi)
=
P_f\!\left(g+\frac{b}{1-a}\rho\right),
\]
and the equilibrium state for the original fiber-dependent potential is the
unique invariant lift of the equilibrium state of the quotient potential
whenever uniqueness holds on the base.

Hence the thermodynamic problem for the two-dimensional skew-product has
been reduced exactly to a one-dimensional expanding map.

\subsection{A symbolic skew-product and weak Gibbs transfer}

We next consider a model in which the invariant measure and the Gibbs
correspondence can be seen directly. Let
\[
M=\{0,1\}^{\mathbb Z}
\]
with the left shift $\sigma$. Let $|a|<1$ and let
$\rho:M\to\mathbb R$ be H\"older. On
\[
N=M\times[-R,R]
\]
consider
\[
F(\omega,y)=\bigl(\sigma\omega,ay+\rho(\omega)\bigr),
\]
with $R$ chosen so that $F(N)\subset N$. Again
\[
\pi(\omega,y)=\omega
\]
has the continuous section $s(\omega)=(\omega,0)$ and the fibers are
uniformly contracted.

Let $g:M\to\mathbb R$ be H\"older and, for
\[
\varphi(\omega,y)=g(\omega)+by,
\]
the same computation as in Proposition~\ref{prop:affine-explicit} gives
\[
u(\omega,y)=\frac{b}{1-a}y,
\qquad
\psi(\omega)=g(\omega)+\frac{b}{1-a}\rho(\omega).
\]

There is also an invariant graph. Define
\[
h(\omega)
=
\sum_{k=1}^{\infty}a^{k-1}\rho(\sigma^{-k}\omega).
\]
The series converges uniformly, and
\[
h(\sigma\omega)=ah(\omega)+\rho(\omega).
\]
Therefore,
\[
F(\omega,h(\omega))
=
(\sigma\omega,h(\sigma\omega)),
\]
so
\[
\Gamma_h=\{(\omega,h(\omega)):\omega\in M\}
\]
is an invariant graph.

Let $\mu$ be an invariant Gibbs measure for the H\"older potential
$\psi$ on the full shift and let
\[
\widetilde\mu=(\omega\mapsto(\omega,h(\omega)))_*\mu.
\]
Then $\widetilde\mu$ is the unique invariant lift of $\mu$ and is supported
on $\Gamma_h$.

\begin{proposition}
\label{prop:symbolic-gibbs-example}
In the symbolic skew-product above, the pair
$(\widetilde\mu,\mu)$ satisfies both forward and backward Bowen-mass
comparability with zero logarithmic mass error (after a change of Bowen-ball
radii). Consequently, if $\mu$ is Gibbs
for $(\sigma,\psi)$, then $\widetilde\mu$ is weak Gibbs for $(F,\varphi)$.
In contrast, weak Gibbs for $(F,\varphi)$ in the invariant graph implies weak
Gibbs for $(\sigma,\psi)$.
\end{proposition}

\begin{proof}
The graph map
\[
\iota(\omega)=(\omega,h(\omega))
\]
is a homeomorphism from $M$ onto $\Gamma_h$ and conjugates $\sigma$ with
$F|_{\Gamma_h}$. Hence, Bowen balls on the graph are carried into Bowen balls
for the shift, with only a change of radius determined by the uniform
continuity of $\iota$ and $\iota^{-1}$. Since
$\widetilde\mu=\iota_*\mu$, no multiplicative measure distortion occurs.
Thus, the logarithmic mass errors can be taken equal to zero; only the Bowen-ball radii are reparametrized.
The weak Gibbs transfer follows from Theorems
\ref{thm:weak-gibbs-forward} and \ref{thm:weak-gibbs-backward}.
\end{proof}

\begin{corollary}
\label{cor:symbolic-sequential-gibbs}
In the symbolic skew-product above, if $\mu$ is Bowen sequential Gibbs for
$(\sigma,\psi)$ along an admissible sequence $n_i(\omega)$, then
$\widetilde\mu$ is Bowen sequential Gibbs for $(F,\varphi)$ and the same
prescribed sequence remains admissible on the invariant graph. Conversely,
Bowen sequential Gibbs on the invariant graph transfers to the base along the
sequence induced by the graph inverse.
\end{corollary}

\begin{proof}
The graph conjugacy has zero logarithmic mass distortion, as proved in
Proposition~\ref{prop:symbolic-gibbs-example}; hence the associated forward
and backward sequential Bowen-mass comparisons are uniformly bounded. Apply
Corollary~\ref{cor:sequential-gibbs-transfer}.
\end{proof}

\begin{remark}
If, in addition, a cylinder sequential Gibbs measure in the sense of
Ferreira--Oliveira satisfies the bounded finite-window distortion hypothesis of
Lemma~\ref{lem:cylinder-bowen-bridge}, then its prescribed cylinder Gibbs-time
sequence is also admissible for the Bowen formulation and the preceding
corollary applies. This extra hypothesis is automatic in several classical
Gibbs/quasi-Bernoulli symbolic settings, but it is not a formal consequence of
sequential Gibbs alone.
\end{remark}

\subsection{A non-invertible base with a non-graph Gibbs lift}
\label{subsec:non-graph-gibbs}

Let $M=\{0,1\}^{\mathbb N_0}$ with the one-sided shift $\sigma$ and
the metric $d_M(x,x')=2^{-k(x,x')}$, where $k$ is the first index of disagreement.
On $N=M\times[0,1]$ use the maximum product metric and define
\[
 F(x,y)=\left(\sigma x,\frac y3+\frac23x_0\right),
 \qquad \varphi(x,y)=t x_0+by,
\]
where $b\ne0$ and $t+b\ne0$. The shift is a non-invertible local
homeomorphism, the section $(x,0)$ is Lipschitz, and the fiber contraction
factor is $1/3$. The affine computation gives
\[
 u(x,y)=\frac{3b}{2}y,\qquad \psi(x)=(t+b)x_0.
\]
The potential $\psi$ is not cohomologous to a constant: its averages on
$0^\infty$ and $1^\infty$ are $0$ and $t+b$.

Set $p=e^{t+b}/(1+e^{t+b})\in(0,1)$, let $\mu_p$ be the Bernoulli
measure with digit-one probability $p$, and let $\nu_p$ be the law of
\[
 Y=\sum_{k=1}^{\infty}\frac{2\xi_k}{3^k},
\]
where the $\xi_k$ are independent Bernoulli digits of parameter $p$.
Write $C=\operatorname{supp}\nu_p$ for the middle-thirds Cantor set.

\begin{proposition}[Uniform Gibbs transfer with non-atomic conditionals]
\label{prop:non-graph-gibbs}
The unique invariant lift of $\mu_p$ is
$\widetilde\mu_p=\mu_p\otimes\nu_p$. Its fiber conditionals are
non-atomic, so this lift is not supported on a measurable graph over $M$.
Both measures satisfy the strong, all-time Bowen Gibbs estimates for their
respective potentials, on $M$ and $M\times C$, with pressure
\[
 P=\log(1+e^{t+b}).
\]
Forward and backward Bowen-mass comparison have errors bounded independently
of time. In particular the blockwise sequential comparisons are uniformly
bounded along every prescribed admissible base sequence.
\end{proposition}

\begin{proof}
Under $\mu_p\otimes\nu_p$, the variables $x_0$, $\sigma x$, and $Y$
are independent. The random variable $Y/3+2x_0/3$ again has law $\nu_p$
and is independent of $\sigma x$. This proves invariance; uniqueness is
Theorem~\ref{thm:lift-bijection}. The coding by Cantor digits is injective,
and a cylinder of $k$ digits has probability at most
$\max(p,1-p)^k\to0$. Thus $\nu_p$ has no atoms. Disintegration of a graph
measure would be Dirac almost everywhere, which excludes graph support.

For $0<\eps<1$ and $n\ge1$, two base points in the same Bowen ball
have equal digits at indices $0,\ldots,n-1$. Their fiber coordinates
therefore differ at time $j<n$ by $3^{-j}|y-y'|$. The maximum metric,
including time zero, gives the exact identity
\[
 B_F((x,y),\eps,n)=B_\sigma(x,\eps,n)\times
 \bigl((y-\eps,y+\eps)\cap[0,1]\bigr).
\]
Consequently, for $y\in C$,
\[
 \widetilde\mu_p(B_F((x,y),\eps,n))
 =\mu_p(B_\sigma(x,\eps,n))\nu_p(B(y,\eps)).
\]
For every fixed $\eps>0$, full support on compact $C$ implies
$c_\eps:=\inf_{y\in C}\nu_p(B(y,\eps))>0$.
Indeed cover $C$ by finitely many balls $B(y_k,\eps/4)$ and use
$B(y_k,\eps/4)\subset B(y,\eps)$ whenever
$y\in B(y_k,\eps/4)$. Thus the mass ratio lies in $[c_\eps,1]$,
which proves both directions of bounded comparison at the same radius.

The base Bowen ball fixes $n+L(\eps)$ initial digits, where the integer
$L(\eps)\ge0$ depends only on the radius. The first $n$ digits have mass
\[
 \exp(S_n\psi(x)-nP).
\]
The remaining $L(\eps)$ digits have mass between
$\min(p,1-p)^{L(\eps)}$ and $1$. These are uniform strong Gibbs bounds.
The cohomological endpoint error is at most $2\|u\|_\infty$, so the
product identity proves the strong Gibbs bounds upstairs. Pressure equality
follows from Theorem~\ref{thm:pressure}. All iterated fiber coordinates
remain in $C$; hence the same bounded comparisons hold at every suffix
center and every suffix length, proving the sequential assertion.
\end{proof}

This example verifies the conditional-mass hypothesis of
Proposition~\ref{prop:uniform-skew-bowen} with a bound independent of $n$:
the conditionals equal $\nu_p$ and $q_n(r,\delta)=-\log c_r$.
The lift is therefore outside the graph setting, and the reduced potential
is not cohomologous to a constant. The base in this example is uniformly
expanding. The later neutral-circle example concerns subactions.

\subsection{Polynomial renewal tails and unbounded forward comparison}
\label{subsec:renewal-unbounded}

The Bernoulli example has bounded comparison. We next consider a non-graph
example in which the uniform forward logarithmic error is $O(\log n)$
and cannot be replaced by a constant, even after fixed changes of Bowen radii. The potential is continuous and non-H\"older. We verify its cohomological
identity directly and then apply the Gibbs transfer theorems.

Fix $\alpha>2$, set $Z=\sum_{k\ge0}(k+1)^{-\alpha}$, and put
\[
 p_k=Z^{-1}(k+1)^{-\alpha},\qquad
 S_k=\sum_{l\ge k}p_l,\qquad
 \mathfrak m=\sum_{k\ge0}(k+1)p_k<\infty.
\]
Construct a stationary two-sided renewal process by concatenating blocks
$10^k$ with independent gap law $p_k$. To specify stationarity explicitly,
the block containing the origin has length-biased law
$(k+1)p_k/\mathfrak m$, the position of the origin in it is uniform,
and the blocks to its left and right have independent law $(p_k)$.
Denote its law by $\widehat\mu$, and its one-sided marginal by $\mu$.
The uniform-position construction is shift invariant: crossing a block
boundary regenerates an independent block with law $(p_k)$, while each
position in a block of type $k$ has mass $p_k/\mathfrak m$.
Every finite binary word has positive mass.

For a word $w$ containing a $1$, let $a,b$ be its initial and terminal
zero-run lengths, and let $k_1,\ldots,k_s$ be the gaps between its successive
ones. The renewal construction gives
\begin{equation}
 \mu([w])=\frac{S_a S_b}{\mathfrak m}\prod_{i=1}^s p_{k_i},
 \qquad
 \mu([0^n])=\frac1{\mathfrak m}\sum_{r\ge n}S_r.
 \label{eq:renewal-cylinders}
\end{equation}
Indeed, the intensity of renewal sites is $1/\mathfrak m$; the boundary
gaps must cover $a$ and $b$ zeros, and the internal gaps are specified.
For an all-zero word, a gap of length $k\ge n$ contributes $k-n+1$
possible positions. In particular, with constants depending only on $\alpha$,
\[
 S_k\asymp(k+1)^{1-\alpha},\qquad
 \mu([0^n])\asymp(n+1)^{2-\alpha}.
\]
Here $A\asymp B$ means that their ratio has positive constant upper and
lower bounds independent of the indices.

On the one-sided full shift define
\[
 g(x)=\begin{cases}
 \log p_0,&x_0=1,\\
 \log(p_k/p_{k-1}),&x\in[0^k1],\quad k\ge1,\\
 0,&x=0^\infty.
 \end{cases}
\]
This potential is continuous: on $[0^k1]$ its value is
$\alpha\log(k/(k+1))\to0$. Its variation at $0^\infty$ decays as
$1/k$, so it is not H\"older for the standard exponential symbolic metric.

\begin{lemma}[Polynomial Gibbs and finite-past bounds]
\label{lem:renewal-polynomial}
The potential $g$ has pressure zero. There exist $A,D>0$ such that
\[
 [A(n+1)^D]^{-1}\le
 \frac{\mu([x_0\ldots x_{n-1}])}{\exp(S_ng(x))}
 \le A(n+1)^D
\]
for every $x$ and $n\ge1$.
For every fixed integer $L\ge1$ there is $c_L>0$ such that every
binary word $v$ of length $L$ and word $w$ of length $n$ satisfy
\begin{equation}
 \frac{c_L}{n+L+1}\mu([w])\le\mu([vw])\le\mu([w]).
 \label{eq:renewal-prepend}
\end{equation}
\end{lemma}

\begin{proof}
First consider a word containing ones. If its terminal run has $b$ zeros,
the contribution of those zeros to $\exp(S_ng)$ is a number $R$ satisfying
$(b+1)^{-\alpha}\le R\le1$, regardless of the continuation beyond the word.
Telescoping the initial run and the complete internal gaps gives
\[
 \exp(S_ng(x))=p_a R\prod_{i=1}^s p_{k_i}.
\]
Thus its mass-to-weight ratio is
$S_a S_b/(\mathfrak m p_a R)$, bounded above and below by powers of $n+1$.
For an all-zero word, $(n+1)^{-\alpha}\le\exp(S_ng(x))\le1$;
\eqref{eq:renewal-cylinders} gives the same polynomial conclusion.

For completeness, pressure can be checked without an equilibrium-state
existence theorem. The cylinder partition sum of the supremum weights is
at least $1$, due to $0^\infty$. For a word whose first $1$ is at position
$a$, the initial-run weight is at most $1$, each complete renewal block
has weight $p_k$, and the final partial block has weight at most $1$ (not necessarily
$p_0$). Indeed, its contribution is a terminal factor $R$ with
$0<R\le1$, and a continuation by zeros can make $R=1$.
For a fixed number of complete blocks, summing their products over the
permitted gap tuples is at most $1$. There are at most $n$ choices of $a$
and $n$ choices for that number. Hence the partition sum is at most
$1+n^2$, proving $P_\sigma(g)=0$.

To prove \eqref{eq:renewal-prepend}, if $w$ contains a $1$ and starts
with $a$ zeros, prepending an all-zero word changes $S_a$ to $S_{a+L}$,
whose ratio is bounded below for fixed $L$. If $v$ contains a $1$, its
last $1$ creates a joining gap $a+d$ with $0\le d<L$; all other factors
involving $v$ belong to a finite list of positive constants. The only
unbounded-index ratio is $p_{a+d}/S_a\asymp_L(a+1)^{-1}$.
If $w=0^n$, an all-zero prepend gives the ratio of two zero-cylinder
masses at lengths $n+L$ and $n$. A prepend containing a $1$ instead gives,
up to positive constants depending on $v$, the ratio
$S_{n+d}/\sum_{r\ge n}S_r\asymp_L(n+1)^{-1}$.
These cases prove the uniform lower bound. The upper bound follows from
stationarity: the event $[vw]$ is contained in $\sigma^{-L}[w]$.
\end{proof}

Use the same contracting map
$F(x,y)=(\sigma x,y/3+2x_0/3)$ and maximum metric as before.
Push $\widehat\mu$ forward by
\[
 \omega\longmapsto\left((\omega_j)_{j\ge0},
       \sum_{k\ge1}\frac{2\omega_{-k}}{3^k}\right)
\]
to obtain $\widetilde\mu$. This coding intertwines the two-sided shift
and $F$, so $\widetilde\mu$ is the unique invariant lift of $\mu$.
For $b\ne0$ put
\[
 u(x,y)=by,\qquad
 \varphi(x,y)=g(x)+\frac{2b}{3}(y-x_0).
\]
Then $\varphi=g\circ\pi+u-u\circ F$ is an explicit continuous,
fiber-dependent potential, and its pressure is zero.

\begin{proposition}[Subexponential comparison without bounded comparison]
\label{prop:renewal-unbounded}
Both $\mu$ and $\widetilde\mu$ are uniform weak Gibbs for their associated
potentials. Their forward and backward Bowen-mass comparisons hold on
$M$ and $M\times C$ with logarithmic errors $O_\eps(\log(n+2))$.
The lift is not supported on a measurable graph. Forward comparison
uniform over a full-measure set cannot have bounded logarithmic error,
even after fixed reparametrizations of the radii. Its optimal error order
is $\log n$. In contrast, backward comparison is uniformly bounded along
the prescribed section $\ell(x)=(x,0)$. Neither measure has the uniform
strong all-time Gibbs property.
\end{proposition}

\begin{proof}
A base Bowen ball is a cylinder of length $n+L_\eps$, with $L_\eps$
fixed. The preceding lemma, together with the bounded extra $L_\eps$ terms in the
Birkhoff sum, gives uniform weak Gibbs estimates at every small radius.
The geometric product identity from
Proposition~\ref{prop:non-graph-gibbs} remains valid, but the measure is
no longer a product. Choose $L$ with $3^{-L}<\eps$. Fixing the first
$L$ past Cantor digits guarantees that the fiber coordinate lies in
$B(y,\eps)$. The joint past--future cylinder has measure $\mu([vw])$,
where $v$ is that past word in chronological order and $w$ is the base
Bowen word. Therefore \eqref{eq:renewal-prepend} gives
\[
 \frac{c_\eps}{n+C_\eps}\mu(B_\sigma(x,\eps,n))
 \le\widetilde\mu(B_F((x,y),\eps,n))
 \le\mu(B_\sigma(x,\eps,n)).
\]
These bounds give both comparisons at the same radius;
the weak Gibbs transfer theorem now applies. For the uniform comparison, fix a measurable full-$\widetilde\mu$-measure
set $G$ of centers in $M\times C$. Define
\[
 A_n(\eps;G)=\sup_{\hat x\in G}
 \log\frac{\mu(B_\sigma(\pi\hat x,\eps,n))}
 {\widetilde\mu(B_F(\hat x,\eps,n))}.
\]
The cylinder lower bound makes each denominator positive. The displayed
comparison implies, uniformly over $G$ and $n\ge1$,
\begin{equation}
 0\le A_n(\eps;G)\le\log(n+C_\eps)-\log c_\eps.
 \label{eq:renewal-optimal-upper}
\end{equation} The same estimates hold at
every suffix center, and yield blockwise errors $O_\eps(\log(n_i+2))$
along any prescribed sequence $n_i\to\infty$.

To rule out bounded forward comparison, fix $0<\eps<1/3$ and choose
centers whose most recent past digit is $1$ and whose next $2n$ digits
are zero. For each $n$, this set has positive measure.
Their fiber neighborhoods require the most recent past digit to remain $1$: in the
middle-thirds Cantor coding, points with first past digit $0$ lie in
$[0,1/3]$, whereas points with first past digit $1$ lie in $[2/3,1]$.
The distance between these sets is at least $1/3>\eps$.
Writing $N=n+L_\eps$, the lifted mass is at most
\[
 \widehat\mu(\omega_{-1}=1,\omega_0\ldots\omega_{N-1}=0^N)
 =\frac{S_N}{\mathfrak m}\asymp(n+1)^{1-\alpha},
\]
whereas the base Bowen mass is $\asymp(n+1)^{2-\alpha}$.
Here $N\le 2n$ once $n\ge L_\eps$, so the chosen event with
$2n$ future zeros indeed determines the required $N$-symbol cylinder.
The ratio is at most $C_\eps/n$. Any other fixed base radius changes
the zero-cylinder length only by a constant, so it cannot restore a
uniform positive ratio. The obstruction persists after removing any null set. Indeed, for each $n$ the event
$E_n=\{\omega_{-1}=1,\ \omega_0=\cdots=\omega_{2n-1}=0\}$
has probability $S_{2n}/\mathfrak m>0$. Its intersection with the
preimage of any full-$\widetilde\mu$-measure set $G$ is nonempty.
At such a center the base ball is a zero-cylinder of length $n+L_\eps$
and the lifted ball is contained in the cylinder imposing the past digit
$1$ and that same future zero-word. Consequently
\[
 A_n(\eps;G)\ge
 \log\frac{\mu([0^{n+L_\eps}])}
 {S_{n+L_\eps}/\mathfrak m}
 \ge\log n-C_\eps'
\]
for all sufficiently large $n$. Together with
\eqref{eq:renewal-optimal-upper}, this proves
$A_n(\eps;G)=\log n+O_\eps(1)$ for every full-measure $G$.
For any other fixed base comparison radius $\delta>0$, choose
$0<\delta_0<\min\{\delta,1/3\}$. By monotonicity,
$\mu(B_\sigma(0^\infty,\delta,n))\ge
\mu(B_\sigma(0^\infty,\delta_0,n))
\asymp_{\delta_0} n^{2-\alpha}$.
The upper bound on the lifted ball remains $O_\eps(n^{1-\alpha})$
at the same positive-measure centers. Hence the obstruction holds
for \emph{every} fixed positive base radius, including radii for which
the Bowen ball is not represented by a nontrivial cylinder.
The constants may depend on $\eps$ and $\delta$.
Thus any uniform forward logarithmic
error is at least $\log n-O_\eps(1)$; together with the upper estimate,
this proves the following precise optimality statement. For each fixed
$0<\eps<1/3$, the lower bound above permits
$a_n^+(\eps)\le\log(n+C_\eps)+O_\eps(1)$. Conversely, for any fixed
positive base comparison radii, any forward error valid uniformly on a
full-measure set satisfies $a_n^+(\eps)\ge\log n-O(1)$ for all
sufficiently large $n$. Constants may depend on these fixed radii, but not
on time or center. This does not concern radii that vary with $n$, nor
center-dependent errors.

Backward comparison has different quantifiers: it uses a prescribed lift
rather than every center in a full-measure set. Choose $\ell(x)=(x,0)$,
which is a continuous section contained in $M\times C$. For fixed $L$,
prepending $0^L$ to a word containing a $1$ changes its initial factor by
$S_{a+L}/S_a$. For an all-zero word the ratio is
$\mu([0^{n+L}])/\mu([0^n])$. The polynomial tail estimates give positive
lower bounds for both ratios, uniform in $a,n$. Choose $L$ with
$3^{-L}<\eps$. The cylinder with $L$ zero past digits lies inside the
fiber ball about $0$, and consequently
\[
 c'_\eps\mu(B_\sigma(x,\eps,n))
 \le\widetilde\mu(B_F(\ell(x),\eps,n))
 \le\mu(B_\sigma(x,\eps,n)).
\]
This proves bounded backward comparison along $\ell$ at the same radius.
The section need not push $\mu$ to an invariant measure: it specifies
centers for the comparison, while $\widetilde\mu$ remains the invariant
lift. This argument does not assert bounded sequential comparison along
$\ell$, which would also require control at $F^j\ell(x)$.

The lift is non-graph: given any future whose first renewal is at a finite
distance $r$, the conditional probabilities of a most recent past digit
$1$ and $0$ are both positive. The first is $p_r/S_r\in(0,1)$, by the
renewal cylinder formulas. Such futures have full measure. Thus the fiber
conditionals are not Dirac.

Finally, choose centers with $2n$ initial zeros. Their $n$-step weights
are bounded below by a positive constant independent of $n$, while their
base Bowen masses are $\asymp n^{2-\alpha}\to0$. For the lift choose in
addition the most recent past digit $1$; $u-u\circ F^n$ is bounded and
the lifted masses are even smaller. All these center sets have positive
measure. No uniform strong Gibbs constant can therefore exist.
\end{proof}

\begin{corollary}[Coexisting equilibrium states in the renewal model]
\label{cor:renewal-equilibria}
The renewal measure $\mu$ and $\delta_{0^\infty}$ are distinct equilibrium
states for $g$. Their unique invariant lifts $\widetilde\mu$ and
$\delta_{(0^\infty,0)}$ are distinct equilibrium states for $\varphi$.
Writing $H(p)=-\sum_{k\ge0}p_k\log p_k<\infty$, one has
\[
 h_\mu(\sigma)=\frac{H(p)}{\mathfrak m},\qquad
 \int g\,d\mu=-\frac{H(p)}{\mathfrak m}.
\]
No assertion is made that these are the only equilibrium states.
\end{corollary}

\begin{proof}
The event that the first $1$ occurs at index $k$ has probability
$S_k/\mathfrak m$, including $k=0$. The event $0^\infty$ has probability
zero, by \eqref{eq:renewal-cylinders}. Hence
\[
 \mathfrak m\int g\,d\mu
 =\log p_0+\sum_{k\ge1}S_k(\log p_k-\log p_{k-1})
 =\sum_{k\ge0}p_k\log p_k.
\]
The last equality follows by summation by parts; its boundary term tends
to zero since $S_k|\log p_k|=O(k^{1-\alpha}\log(k+1))\to0$.
The series converge, since $\sum p_k|\log p_k|<\infty$.

Let $\mathcal P$ be the partition by the first digit. Integrating the
logarithmic polynomial Gibbs bounds in
Lemma~\ref{lem:renewal-polynomial} gives
\[
 \left|H_\mu\!\left(\bigvee_{j=0}^{n-1}\sigma^{-j}\mathcal P\right)
       +n\int g\,d\mu\right|\le \log A+D\log(n+1).
\]
The first-digit partition generates the one-sided shift Borel sigma-algebra.
Divide by $n$ and take the limit to obtain
$h_\mu(\sigma)=-\int g\,d\mu$. Since $P_\sigma(g)=0$, the renewal
measure is an equilibrium state. The fixed-point measure has entropy zero
and $g$-average zero, so it is an equilibrium state as well. They are
distinct because $\mu([1])=1/\mathfrak m>0$.
The correspondence in Theorem~\ref{thm:pressure} gives the stated lifts.
\end{proof}

The same model thus requires a growing error for uniform forward
comparison but admits a bounded error along a backward section. Its weak
Gibbs estimates hold at every time. The coexistence of equilibrium states
also shows that these estimates alone do not ensure uniqueness for
continuous potentials.

\subsection{Counterexample: pointwise errors cannot replace uniform errors}
\label{subsec:counterexample-uniform}

\begin{proposition}[Pointwise weak Gibbs without uniform weak Gibbs]
\label{prop:pointwise-not-uniform}
There is a compact symbolic homeomorphism with zero topological entropy and an
invariant probability that has pointwise subexponential Gibbs errors for
the zero potential on a full-measure set, but has no uniform weak Gibbs
estimates on any full-measure set. The same failure occurs in a uniformly
fiber-contracting extension with zero relative entropy.
\end{proposition}

\begin{proof}
Let $X\subset\{0,1\}^{\mathbb Z}$ be the closure of the union of the
bilateral shift orbits of the two-sided periodic sequences with repeating
block $0^k1^k$, $k\ge1$. Use the metric
$d_X(x,x')=2^{-\min\{|j|:x_j\ne x'_j\}}$ for distinct points.
The compact set $X$ is invariant under both the shift and its inverse;
hence $\sigma|_X$ is a homeomorphism and in particular a local
homeomorphism. There are $O(n^2)$ words of length $n$:
periods with $k<n$ contribute at most $\sum_{k<n}2k$, and those with
$k\ge n$ contribute only words with at most one transition. Closure adds
no finite words. Hence $h_{\mathrm{top}}(\sigma|_X)=0$ and the zero
potential has pressure zero. Let $\eta_k$ be uniform on the period-$2k$
orbit and put
\[
 \eta=\sum_{k\ge1}w_k\eta_k,\qquad
 w_k=\frac{e^{-k^2}}{\sum_{l\ge1}e^{-l^2}}.
\]
Every periodic point in this full-measure union is an atom of mass
$w_k/(2k)$. Every Bowen ball centered there contains that atom, so
pointwise Gibbs bounds hold with the time-independent, center-dependent
constant $2k/w_k$.

In contrast, at a center whose coordinates $0,\ldots,2k-1$ form
$0^k1^k$, a Bowen ball of length
$2k$ and radius $\eps<1$ is contained in $[0^k1^k]$. This cylinder can
occur only in orbits indexed by $l\ge k$, at one position per period.
Hence
\[
 \eta(B_\sigma(x,\eps,2k))\le
 \sum_{l\ge k}\frac{w_l}{2l}\le C e^{-k^2}.
\]
Uniform weak Gibbs constants would satisfy
$\log K_{2k}(\eps)\ge k^2-\log C$, contradicting
$\log K_n(\eps)=o(n)$. Every full-measure set contains every positive-mass
atom, so the obstruction cannot be discarded.

For the extension, take $N=X\times[0,1]$,
$F(x,y)=(\sigma x,y/2)$ and $\widetilde\eta=\eta\otimes\delta_0$.
The projection has a continuous section, fibers contract uniformly, and
Bowen masses on the support equal those of the base. Thus relative entropy
is zero but the same failure of uniform weak Gibbs persists. The base is
a compact homeomorphism, the selector $(x,0)$ is continuous, and the
fiber contraction factor is $1/2$. Thus this counterexample satisfies the
geometric hypotheses of Section~\ref{ass:selector}.
\end{proof}

The renewal example shows that bounded forward comparison, uniform in
the lifted centers, cannot be deduced from fiber contraction and uniform
weak Gibbs estimates. Its bounded backward comparison along a chosen
section highlights the distinct quantifiers in the two definitions. The counterexample addresses uniformity: pointwise subexponential
constants need not admit the uniform envelope required by the all-time
transfer theorems.

\subsection{A supplementary critical subaction over a neutral quotient}
\label{subsec:zooming-application}

\begin{remark}[An elementary critical example]
\label{thm:zooming-extension}
On $M=\mathbb R/\mathbb Z$ set
\[
 f(x)=2x-\frac{\sin(2\pi x)}{2\pi}\pmod1,
 \qquad h(x)=\frac{1-\cos(2\pi x)}2.
\]
This is a smooth degree-two local diffeomorphism with derivative
$2-\cos(2\pi x)\ge1$, strictly greater than $1$ off $0$. Since
$(f^n)'(0)=1$, it is not uniformly expanding. For $0<a<1$, $q>0$,
and $b\ne0$, the map and potential on $M\times[0,1]$
\[
 F(x,y)=(f(x),ay+(1-a)h(x)),\qquad
 \varphi(x,y)=(q-b)h(x)+by
\]
have the direct identity
$\varphi=qh\circ\pi+u-u\circ F$, $u(x,y)=by/(1-a)$.
Here $F$ preserves the cylinder, contracts its fibers by $a$, and admits
the section $(x,0)$. Nonnegativity of $qh$ gives
$m(\varphi,F)=0$ and $\varphi\ge u-u\circ F$; since $h$ vanishes
only at $0$, the minimizing measures are $\delta_0$ and
$\delta_{(0,0)}$.

The reduced potential is not cohomologous to a constant. Indeed, with
$\widetilde f$ the displayed lift, the function
$(\widetilde f\circ\widetilde f)(t)-t$ goes from $0$ to $3$ on
$[0,1]$, so it takes the value $1$ at an interior point of period two.
The only fixed point is $0$, because
$t-\sin(2\pi t)/(2\pi)$ increases strictly from $0$ to $1$.
Thus the period-two average of $qh$ is positive while its fixed-point
average is zero. This elementary illustration uses a nonnegative reduced
potential; it is not a general critical subaction theorem or a Gibbs
application for the neutral quotient.
\end{remark}

\paragraph{Conditional geometric classes.}
Solenoids \cite{Williams} and partially hyperbolic attractors are covered
only after the section, contraction, potential, and mass hypotheses are
verified in the chosen quotient representation. Non-uniformly expanding
quotient theories \cite{OliveiraViana,Pinheiro,VarandasViana,
AlvesRamosSiqueira} can likewise be combined with the transfer statements
when their hypotheses apply. 

\section{Relation with Previous Work}\label{sec:comparison}

The cohomological and variational ingredients are classical. The question here is quantitative: how Bowen-ball masses compare between a contracting extension and its quotient.

\subsection{Classical variational tools and extension theory}

Stable normalization is related to Sinai's cohomology lemma
\cite{Sinai}. The entropy and pressure statements use zero fiberwise
entropy and the Ledrappier--Walters principle \cite{LedrappierWalters}.
Equilibrium states for partially hyperbolic systems have been studied
under other geometric hypotheses
\cite{AlvesBonattiViana,CastroNascimento,FisherOliveira,
RamosSiqueira,RiosSiqueira,ClimenhagaPesinZelerowicz,
ParmenterPollicott}.

\subsection{Weak Gibbs measures and the overlap with Ferreira--Ramos}

Subexponential weak Gibbs distortion already appears in Yuri \cite{Yuri1999,Yuri2000}. The closest antecedent is Ferreira--Ramos \cite{FerreiraRamos}, whose Section~7 already treats
cohomologous stable-constant potentials, invariant lifts, weak Gibbs
transfer, and equilibrium states for semiconjugate attractors. That section works with homeomorphisms of
attractors and the specified comparison structure between fibers. Here
$F$ need only be continuous, the quotient is a local homeomorphism, and
the continuous global selector is imposed explicitly for the general
potential construction. Gibbs transfer then assumes the stated comparison
for the particular invariant-measure pair. The table summarizes the differences.

\begin{center}
\small
\begin{tabular}{p{0.20\textwidth}p{0.34\textwidth}p{0.34\textwidth}}
\hline
Question & Ferreira--Ramos, Section 7 & Present formulation\\
\hline
Potential & Proposition 7.1 normalizes the stable direction under its
stated hypotheses & Auxiliary stable-series proposition with explicit
section and separate regularity hypotheses\\[3pt]
Invariant lifts & Lemma 7.2 treats uniqueness for the ergodic base measure
used there & Bijection for all invariant measures, proved by a coupling\\[3pt]
Weak Gibbs & Theorem 7.1 asserts transfer in both directions in its
attractor setting & Directional transfer under explicit, pair-dependent
subexponential mass comparisons\\[3pt]
Equilibrium & Lemma 7.3 gives equilibrium correspondence & Equality of
free energies for every corresponding invariant pair\\[3pt]
Sequential estimates & No counterpart established in that section &
Blockwise transfer and bounded-error preservation of prescribed sequences\\
\hline
\end{tabular}
\end{center}

Section~\ref{subsec:non-graph-gibbs} verifies the comparison in the
Bernoulli--Cantor model with bounded error. In the renewal model, the forward error is logarithmic. 

\subsection{Sequential Gibbs measures and factor maps}

Gibbs properties under symbolic factor maps were studied for full shifts by
Kempton--Pollicott and by Chazottes--Ugalde
\cite{KemptonPollicott,ChazottesUgalde}; for subshifts of finite type see
Kempton \cite{KemptonSFT}. H\"older regularity of projected Gibbs states under
fiber-wise mixing factors was sharpened by Piraino \cite{Piraino}, while an
almost-additive extension was developed by Yayama \cite{YayamaAlmostAdditive}.

Sequential Gibbs measures were introduced by Ferreira and Oliveira in a
one-sided symbolic setting designed to capture Gibbs estimates available only
at point-dependent times \cite{FerreiraOliveira}. Their theorem gives a
non-uniform counterpart of the preceding factor-Gibbs results: under their
regularity condition on a one-block factor map, the image of a sequential
Gibbs measure is again sequential Gibbs, with the same maximal Gibbs-time
sequence in their cylinder framework. Section~\ref{sec:sequential-gibbs} instead transfers sequential Gibbs estimates from the quotient to the contracting extension, provided Bowen-ball masses are comparable along the prescribed blocks. Subexponential comparison yields blockwise Bowen sequential weak Gibbs,
while uniformly bounded comparison preserves the strong Bowen sequential
Gibbs property along the prescribed admissible sequence. In symbolic settings,
comparison with the cylinder formulation of Ferreira--Oliveira additionally
uses Lemma~\ref{lem:cylinder-bowen-bridge}.

\subsection{Continuous renewal potentials and Gibbs inequalities}

Hofbauer's examples \cite{Hofbauer} established nonuniqueness of equilibrium
states for continuous potentials. Fern\'andez, Gallo, and Maillard
\cite{FernandezGalloMaillard} constructed a regular $g$-measure with a
visible renewal structure that fails to be Gibbsian in the two-sided
specification sense. Their question concerns conditional laws rather than Bowen-ball mass comparison between systems. The distinction in our
renewal model is the direction-dependent $\log n+O(1)$ bound,
including its optimality on every full-measure set. The renewal potential in
Section~\ref{subsec:renewal-unbounded} is of the same classical non-H\"older
kind: its values telescope over renewal blocks and its variations decay
polynomially rather than exponentially. The coexistence of renewal and fixed-point equilibrium measures is also classical. The logarithmic estimate here follows from finite-past versus future cylinder masses.

Ranu \cite{Ranu} proves an almost-everywhere very weak Gibbs inequality for
ergodic equilibrium states of continuous symbolic potentials and gives a
unique-equilibrium example where a Gibbs ratio grows exponentially at a
fixed point. These results concern Gibbs ratios within a single system. Our renewal estimates compare two systems uniformly in the centers;
the positive-measure obstruction persists on any full-measure set.
The atomic example instead concerns uniformity for a nonergodic measure.

\subsection{Consequences for quotient cohomology and optimization}

The potential identity
\[
    \varphi=\psi\circ\pi+u-u\circ F
\]
provides a direct way to lift subactions from the quotient.
If
\[
    \psi\geq v-v\circ f,
\]
then $U=u+v\circ\pi$ satisfies
\[
    \varphi\geq U-U\circ F.
\]
The same formula lifts the cohomological equation
$\psi=v-v\circ f$ to $\varphi=U-U\circ F$.

Mbarki--Santana~\cite{MbarkiSantana} prove subcohomology
and Liv\v{s}ic results for zooming systems under their
stated hypotheses. Here these results can be applied
through the quotient to observables that also depend
on the contracting coordinate. The extension need not
itself satisfy the zooming hypotheses. What must be
checked is that the quotient dynamics and the reduced
potential satisfy the assumptions of the chosen theorem.

For H\"older quotient theorems, regularity of the reduced potential
can be obtained from Proposition~\ref{prop:holder-reduction} or an
explicit formula. The lifted subaction is H\"older whenever
$u$ and $v\circ\pi$ are H\"older.

The Liv\v{s}ic conclusions also depend on the hypothesis
used on the quotient. Theorem~\ref{thm:livsic} requires
a quotient theorem based on vanishing periodic sums.
Corollary~\ref{cor:two-sided-sc-livsic} uses a different
route: subcohomology for $\psi$ and $-\psi$, vanishing
averages for every invariant probability measure, and
density of periodic points. This is the form related
to the Liv\v{s}ic conclusion
in~\cite{MbarkiSantana}. Other quotient theorems can
be used when their assumptions
hold~\cite{Livsic,deLaLlaveMarcoMoriyon}.

For ergodic optimization, the affine invariant-measure correspondence
and matching averages give the minimizing-measure correspondence of
Theorem~\ref{thm:ergodic-optimization}. This is consistent with the
classical cohomological approach to optimization
\cite{BouschJenkinson,GaribaldiBook,GaribaldiLopesThieullen}.

\section*{Acknowledgements}

The author thanks the anonymous referee for helpful comments that improved the exposition.

This work was supported by PRAPG\_CAPES-BRASIL. 

\end{document}